\documentclass{article}
\usepackage{graphicx} 
\usepackage{enumerate}
\usepackage{tabularx}
\usepackage{adjustbox}
\usepackage[utf8]{inputenc}
\usepackage{amsmath}
\usepackage{amssymb}
\usepackage[margin=1.0in]{geometry}
\usepackage{comment}
\usepackage{hyperref}
\newtheorem{theorem}{Theorem}
\newtheorem{corollary}[theorem]{Corollary}

\newtheorem{proposition}[theorem]{Proposition}

\usepackage{lipsum}
\usepackage{natbib}
\usepackage{listings}
\usepackage{color}
\usepackage{setspace} 
\usepackage{pdflscape}
\newenvironment{proof}{\paragraph{Proof:}}{\hfill$\square$}

\allowdisplaybreaks

\newcommand{\Du}{\operatorname{D}_u}
\newcommand{\Dx}{\operatorname{D}_x}
\newcommand{\rcp}[1]{\frac{1}{#1}} 
 
\newcommand{\cE}{{\mathcal E}}
\newcommand{\cG}{{\mathcal G}} 
\newcommand{\cU}{{\mathcal U}}
\newcommand{\pdiff}[3]{\frac{\partial^{#3}#1}{\partial #2^{#3}}} 
\newcommand{\bp}{\begin{proposition}} 
\newcommand{\ep}{\end{proposition}} 
\newcommand{\bpf}{\begin{proof}} 
\newcommand{\epf}{\end{proof}} 

\title{Combinatorial comparison of general galled trees, \\ time-consistent galled trees, \\ and simplex time-consistent galled trees}
\author{Lily Agranat-Tamir\thanks{Department of Biology, Stanford University, Stanford, CA 94305 USA}\hspace{.15cm} \\
Michael Fuchs\thanks{Department of Mathematical Sciences, National Chengchi University, Taipei, Taiwan} \\
Bernhard Gittenberger\thanks{Department of Discrete Mathematics and Geometry, Technische Universit\"{a}t Wien, Austria} \\
Noah A.~Rosenberg$^*$ \\
Karthik V.~Seetharaman$^*$}
\date{December 16, 2025}

\begin{document}
\onehalfspacing
\maketitle

\noindent{\bf Abstract.} Rooted binary phylogenetic networks are extensions of rooted binary trees, adding reticulation nodes that are designed to represent evolutionary processes that involve hybridization events. Enumerative combinatorics studies have counted leaf-labeled phylogenetic networks in a variety of classes, finding that when the number of reticulations is fixed, the time-consistent galled trees are asymptotically less numerous than each of several network classes that had been previously examined. Here we provide enumerative results on two additional network classes: general galled trees and simplex time-consistent galled trees. We show that for a fixed number of galls, as the number of leaves goes to infinity, the asymptotic count of general galled trees is identical to that of time-consistent galled trees, whereas the count of simplex time-consistent galled trees is smaller. If the number of galls is not restricted, then the asymptotic approximations all differ: simplex time-consistent galled trees are less numerous than time-consistent galled trees, which are in turn less numerous than general galled trees. We also report a variety of additional results: recursions to count the studied networks with small numbers of leaves a fixed number of galls, as well as enumerative results for unlabeled networks in the classes that we investigate.



\clearpage
\section{Introduction}

Rooted binary phylogenetic networks are rooted binary trees with the addition of nodes of in-degree 2 and out-degree 1, \emph{reticulation} or \emph{hybrid nodes}, which produce \emph{reticulation cycles}. These networks are graph structures that are used to study evolutionary processes in which biological lineages that have previously diverged come back together (reviewed by \cite{KongEtAl22}).

Many classes of phylogenetic networks have had their enumerative properties studied mathematically~\citep{AgranatTamirEtAl24b, AgranatTamirEtAl25, AgranatTamirEtAl24a, BienvenuEtAl21, BouvelEtAl20, CardonaAndZhang20, ChangAndFuchs24, ChangEtAl24, FuchsAndGittenberger25, FuchsEtAl19, FuchsEtAl21, FuchsEtAl22b, FuchsEtAl25,  FuchsEtAl22a, FuchsEtAl21b, GUNAWAN2020644, Mansouri22, MathurAndRosenberg23, Zhang19}. In recent work, we have focused on one such class, the \emph{time-consistent galled trees}~\citep{AgranatTamirEtAl24b, AgranatTamirEtAl25,  AgranatTamirEtAl24a, MathurAndRosenberg23}. In \emph{galled trees}, each reticulation cycle shares no vertices or edges with other such cycles, and is known as a \emph{gall}~\citep{AgranatTamirEtAl24b, AgranatTamirEtAl25, MathurAndRosenberg23}. In \emph{time-consistent} galled trees, the two nodes that are immediately ancestral to a reticulation node---its \emph{hybridizing nodes}---do not have an ancestor--descendant relation, so that they can be viewed as contemporaneous if a sense of time is imposed on the network. \emph{Time-consistent} galled trees are equivalent to \emph{normal} galled trees.

A hallmark property in enumerative combinatorics, namely the asymptotic approximation of the size of a combinatorial class as a suitable variable grows in size, has been studied in several classes of phylogenetic networks: reticulation-visible networks, tree-child networks, normal networks, galled networks, galled tree-child networks, and time-consistent galled trees, as well as in the class of general phylogenetic networks. Specifically, the numbers of leaf-labeled general phylogenetic networks \citep{Mansouri22}, reticulation-visible networks \citep{ChangAndFuchs24}, tree-child networks \citep{FuchsEtAl22b}, normal networks \citep{FuchsEtAl22b, FuchsEtAl25}, galled networks \citep{ChangAndFuchs24}, and galled tree-child networks \citep{ChangEtAl24}, all with a fixed number $k$ of reticulation nodes, as the number of leaves $n$ goes to infinity, have the same asymptotic growth. Denoting the counts for these various classes by PN, RV, TC, N, GN, and GTC, respectively, as $n \rightarrow \infty$, the enumerations satisfy:
\begin{equation}
\label{eq:1}
PN_{n,k}\sim RV_{n,k} \sim TC_{n,k} \sim N_{n,k} \sim GN_{n,k} \sim GTC_{n,k} \sim \frac{2^{k-1}\sqrt{2}}{k!}n^{n+2k-1}\Big( \frac{2}{e} \Big)^n.
\end{equation}

We found, however, that the corresponding approximation for the number of leaf-labeled time-consistent galled trees with a fixed number of galls---equivalent to the number of reticulation nodes---differs. In particular, the subexponential part of the asymptotic expansion is smaller by a factor of $(2k)!/(2^k k!)$~\citep[eq.~(53)]{AgranatTamirEtAl25}. Denoting the count of time-consistent galled trees by C,
\begin{equation}
\label{eq:2}
    C_{n,k} \sim \frac{2^{2k-1}}{(2k)! \, \sqrt{\pi}} n^{2k-\frac{3}{2}} \Big( \frac{1}{2} \Big)^{-n} n! \sim \frac{2^{2k-1}\sqrt{2}}{(2k)!}n^{n+2k-1}\Big( \frac{2}{e} \Big)^n.
\end{equation}

In this study, we consider the enumerative properties for two additional classes of galled trees: one that contains the time-consistent galled trees as a proper subset, and one that is itself a proper subset of the time-consistent galled trees. For the galled trees in general, or \emph{general galled trees}---which contain the time-consistent galled trees---does the asymptotic enumeration follow the larger eq.~(\ref{eq:1}), like larger classes of phylogenetic networks, or the smaller eq.~(\ref{eq:2}), like the time-consistent galled trees? For the class known as the \emph{simplex time-consistent galled trees}, which are contained among the time-consistent galled trees, does the asymptotic enumeration follow the time-consistent galled trees as in eq.~(\ref{eq:2}), or is it still smaller?

In Section \ref{sec:ulggt}, we count unlabeled general galled trees, and in Section \ref{sec:llggt}, we count leaf-labeled general galled trees. In Section \ref{sec:ul1tcgt}, we count unlabeled simplex time-consistent galled trees, and in Section \ref{sec:ll1tcgt}, we count leaf-labeled simplex time-consistent galled trees. In particular, we analyze the asymptotic approximations of the counts of these networks as the number of leaves goes to infinity, either with a fixed number of galls (Sections \ref{sec:fixedulggt} and \ref{sec:fixedllggt} for general galled trees and Sections \ref{sec:fixedul1tcgt} and \ref{sec:fixedll1tcgt} for simplex time-consistent galled trees) or with any number of galls (Sections \ref{sec:allulggt} and \ref{sec:allllggt} for general galled trees and Sections \ref{sec:allul1tcgt} and \ref{sec:allll1tcgt} for simplex time-consistent galled trees). Section \ref{sec:allllggt}, covering all numbers of galls in leaf-labeled general galled trees, repeats analysis conducted by \cite{BouvelEtAl20} and \cite{FuchsAndGittenberger25}; the latter study in addition contains results for the unlabeled case of Section \ref{sec:allulggt}. Finally, in Section \ref{sec:exact}, we find recursions that we use to obtain the exact numbers of the various networks for small numbers of leaves. 

The analysis confirms that general galled trees with a fixed number of galls behave asymptotically in the same way as time-consistent galled trees with a fixed number of  galls. However, simplex time-consistent galled trees with a fixed number of galls behave differently. The three classes of networks differ asymptotically when no restrictions are imposed on the number of galls, that is, when looking at the union of networks with fixed numbers of galls.

\section{Definitions} \label{sec:defs}

The study builds from \cite{AgranatTamirEtAl25} and uses the same terms, where appropriate. Many definitions for phylogenetic networks were provided by \cite{KongEtAl22}. We consider``non-plane'' networks, in which changing the left--right order in which nodes are graphically depicted does not change the network.

A \emph{rooted binary phylogenetic network} is a simple directed rooted graph with the following four types of nodes: (1) a unique \emph{root} node with in-degree 0 and out-degree 2; (2) \emph{tree} nodes with in-degree 1 and out-degree 2; (3) \emph{leaf} nodes with in-degree 1 and out-degree 0; (4) \emph{reticulation} nodes with in-degree 2 and out-degree 1. As a simple graph, a phylogenetic network has no edges that connect a node to itself, and each pair of nodes connected by an edge are connected by precisely one edge. A rooted binary tree possesses the first three types of nodes; a phylogenetic network can also have the fourth type. We sometimes refer to the rooted binary phylogenetic networks as \emph{general phylogenetic networks}.

The direction of the network is from the root to the leaves. We sometimes refer to reticulation nodes as \emph{hybrid} nodes; the two parents of a hybrid node are called \emph{hybridizing} nodes. A \emph{reticulation cycle} is a pair of directed paths with common start and end nodes that are otherwise node-disjoint. 

\begin{figure}[tb]
    \centering
    \includegraphics[width=13cm]{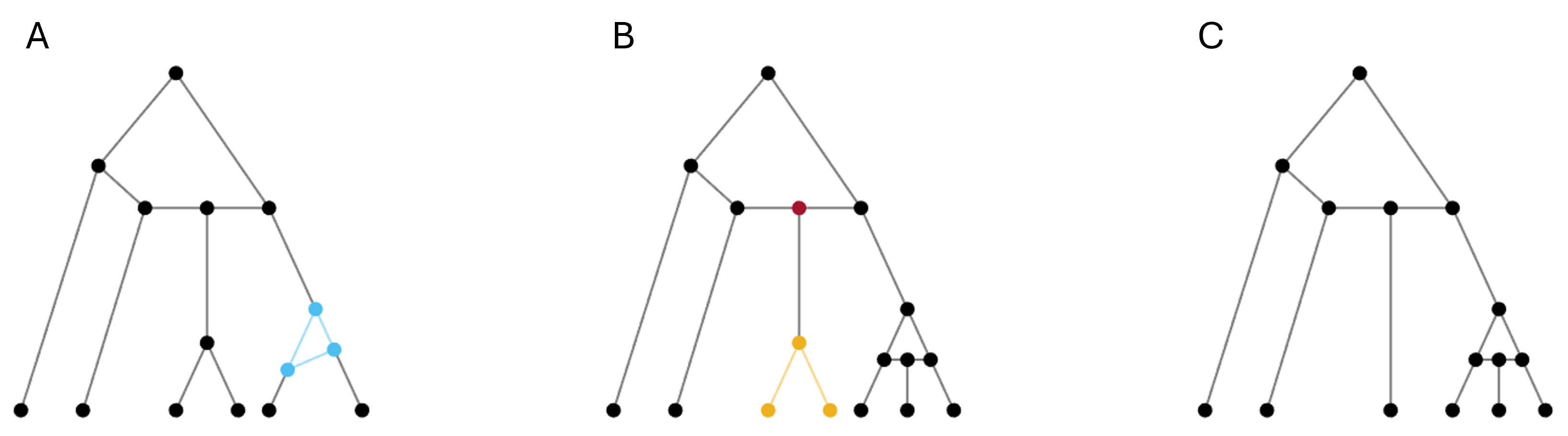}
    \vspace{-.2cm}
    \caption{Three classes of phylogenetic networks. (A) A general galled tree. This galled tree is not a time-consistent galled tree because it has a gall (light blue) whose top node is not separated from the reticulation node by at least two edges (it is a parent of the reticulation node). (B) A time-consistent galled tree. This tree is not a simplex galled tree because it has a reticulation node (red) that is the ancestor of a subtree that is  not a single leaf (orange). (C) A simplex time-consistent galled tree. }
    \label{fig:definitions}
\end{figure}

\emph{Galled trees} are rooted binary phylogenetic networks with the additional property that each node is found in at most one reticulation cycle (Figure \ref{fig:definitions}A). In galled trees, reticulation cycles are called \emph{galls}. A gall possesses an ancestor node (the ``top'' node), from which the two directed paths in the gall originate, and a hybrid node where these paths intersect.

Technically, galled trees that possess reticulation nodes are not actually trees in the mathematical sense; in accord with the term ``galled trees,'' we refer to the galled trees rooted at internal nodes of a galled tree as \emph{subtrees}; these subtrees also need not be trees in a technical sense. For clarity, we sometimes refer to galled trees as \emph{general galled trees}.

\emph{Time-consistent galled trees} are galled trees with the additional property that the ancestor node $r$ of a gall, from which the two disjoint paths reach the reticulation node $a_r$, is separated from $a_r$ by two or more edges  (Figure \ref{fig:definitions}B). This last condition encodes the requirement that we consider only \emph{normal} galled trees; \emph{normal} galled trees are equivalent to \emph{time-consistent} galled trees. The definition of time-consistent galled trees follows d \cite{AgranatTamirEtAl24b},  \cite{AgranatTamirEtAl24a}, and \cite{MathurAndRosenberg23}.

\emph{Simplex galled trees}, also known as \emph{simplicial galled trees} or \emph{1-component galled trees}, are galled trees for which the subtree of each reticulation node is a single leaf (Figure \ref{fig:definitions}C; \cite{FuchsAndGittenberger25, FuchsANDSteel25, Zhang2022}). Henceforth we refer to this class of trees with  the term \emph{simplex} galled trees. Simplex galled trees are suited to scenarios with recent hybridization events, after which insufficient time has passed for lineages descended from reticulation nodes to give rise to further reticulations. 

For time-consistent galled trees, viewing each gall as a representation of a biological merging event, for each gall we depict its hybridizing nodes and hybrid node on a horizontal line. This depiction represents the simultaneity of the hybridizing and hybrid nodes when a galled tree is viewed as a structure evolving in time. Such a depiction is possible in time-consistent networks but not in galled trees for which the two parents of a reticulation node have a parent--child relation (Figure \ref{fig:definitions}A). 

In the network classes that we study---general galled trees (G), time-consistent galled trees (C), and simplex time-consistent galled trees (SC)---we consider both unlabeled and leaf-labeled networks. We stated the definition of general phylogenetic networks (PN) above; although we do not need formal definitions of reticulation-visible networks (RV), tree-child networks (TC), normal networks (N), galled networks (GN), and galled tree-child networks (GTC), it is useful to specify inclusion relationships: the set of galled trees G is a subset of the classes GTC, GN, TC, RV, and PN. The set of time-consistent galled trees C is a subset of G and of N. The set of simplex time-consistent galled trees SC is a subset of C.

\section{Unlabeled general galled trees} 
\label{sec:ulggt}

Note that whereas the number of galls $g$ for time-consistent galled trees $n$ satisfies $0 \leq g \leq \lfloor \frac{n-1}{2} \rfloor$ (Section 2.4 of \cite{AgranatTamirEtAl24a}), for general galled trees, the number of galls satisfies $0 \leq g \leq n-1$. This result can be obtained by noting that the smallest general galled tree has a root gall in which two nodes descend from the root, one of which is a reticulation node, and from both of which a single leaf descends ($n=2$, $g=1$). Each subsequent gall replaces a leaf with a minimum of two leaves; a galled tree with $n$ leaves and $n-1$ galls can be constructed by sequentially replacing leaves with triangular galls from which two leaves descend, so that a general galled tree with $n$ leaves has at most $n-1$ galls.

\subsection{No galls} 
\label{sec:ggt0galls}

The unlabeled general galled trees with no galls are the unlabeled binary trees. The generating function $\mathcal{U}(t) = \sum_{n=0}^\infty U_n t^n$, whose term $U_n$ counts unlabeled binary trees with $n$ leaves, satisfies~\citep[][p.~55]{Harding71, Otter48}, 
$$\mathcal{U}(t) = t + \frac{1}{2}[\mathcal{U}(t^2) + \mathcal{U}(t)^2].$$

Asymptotically, as $t \rightarrow \rho^{-}$, the generating function has equivalence 
\begin{equation}
    \mathcal{U}(t) \sim 1 - \gamma \Big(1 - \frac{t}{\rho}\Big)^{\frac{1}{2}},
    \label{eq:utasymp}
\end{equation} 
where $\rho \approx 0.40270$ and $\gamma \approx 1.13003$~\citep[pp.~476-477]{FlajoletAndSedgewick09}. Further, as $n\rightarrow \infty$,
\begin{equation}
U_n = [t^n] \mathcal{U}(t) \sim \frac{\gamma}{2 \sqrt \pi} n^{-\frac{3}{2}} \rho^{-n}.
\label{eq:Wedderburn}
\end{equation}
The $U_n$ are the Wedderburn--Etherington numbers (OEIS A001190). The first terms appear in Table \ref{table:Etildng}. 

\subsection{One gall} 
\label{sec:ggt1gall}

In the case of one gall, if the root is not part of the gall, then as in the time-consistent case \citep{AgranatTamirEtAl25}, the root has two children, one of which has one gall and the other of which is a binary tree (component 1 in eq.~(\ref{eq:ulsmonegall}) and Figure \ref{fig:ggtonegall}). If the root is part of the gall, then two sequences of trees lead from the root to the reticulation node. Either both sequences of trees are nonempty, as in time-consistent galled trees (component 2), or one is nonempty and the other is empty (component 3). 

\begin{figure}[tb]
    \centering
    \includegraphics[width=6cm]{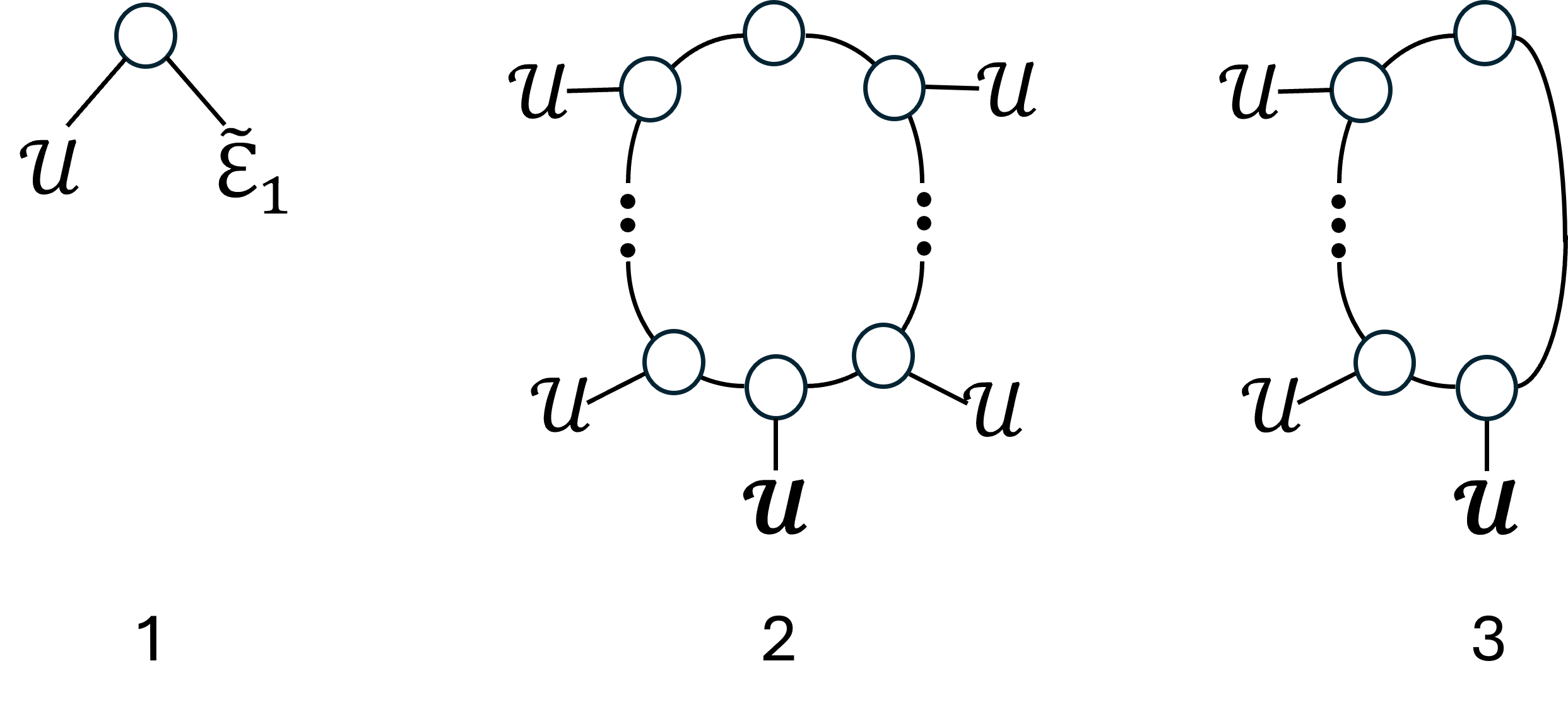}
    \vspace{-0.3cm}
    \caption{The three structures possible for a (non-plane) unlabeled general galled tree with exactly one gall.}
    \label{fig:ggtonegall}
\end{figure}

We follow the symbolic method for enumerating unlabeled structures, as used by \cite{AgranatTamirEtAl25} for unlabeled time-consistent galled trees. In symbolic notation~\citep[p.~93]{FlajoletAndSedgewick09}, 
\begin{equation}
\mathcal{\tilde{E}}_1 = \{\circ\} \times 
\big[\underbrace{\mathcal{U} \times \mathcal{\tilde{E}}_1}_{1} \textrm{   }
\dot{\cup} \textrm{   } \underbrace{\mathcal{U}\times\text{MSET}_2\big(\text{SEQ}^+(\mathcal{U})\big)}_{2}   
\textrm{   } \dot{\cup} \textrm{   } \underbrace{\mathcal{U} \times \text{SEQ}^+(\mathcal{U})}_{3}  \big].
\label{eq:ulsmonegall}
\end{equation}
Converting to a generating function $\mathcal{\tilde{E}}_1(t) = \sum_{n=0}^\infty \tilde{E}_{n,1} t^n$, 
$$\mathcal{\tilde{E}}_1(t) = \mathcal{U}(t) \, \mathcal{\tilde{E}}_1(t) + \frac{\mathcal{U}(t)}{2} \bigg[\bigg(\frac{\mathcal{U}(t)}{1 - \mathcal{U}(t)}\bigg)^2 + \frac{\mathcal{U}(t^2)}{1 - \mathcal{U}(t^2)}\bigg] + \frac{\mathcal{U}(t)^2}{1 - \mathcal{U}(t)}.$$
Solving for $\mathcal{\tilde{E}}_1(t)$, we obtain 
\begin{equation}
\mathcal{\tilde{E}}_1(t) = \frac{\mathcal{U}(t)^3}{2[1-\mathcal{U}(t)]^3} + \frac{\mathcal{U}(t) \, \mathcal{U}(t^2)}{2[1-\mathcal{U}(t)] \, [1-\mathcal{U}(t^2)]} + \frac{\mathcal{U}(t)^2}{[1 - \mathcal{U}(t)]^2}.
\end{equation}

In comparison with $\mathcal{E}_1(t) = \sum_{n=0}^\infty E_{n,1} t^n$, the generating function for the unlabeled time-consistent galled trees with one gall~\citep[eq.~(6)]{AgranatTamirEtAl25}, the generating function for the unlabeled general galled trees with one gall contains only the extra term $\mathcal{U}(t)^2 / [1 - \mathcal{U}(t)]^2$ from component 3 in eq.~(\ref{eq:ulsmonegall}). The extra term is not a term of highest order; asymptotically, from eq.~(\ref{eq:utasymp}), as $t \rightarrow \rho^-$, $\mathcal{\tilde{E}}_1(t) \sim \mathcal{E}_1(t)$. 

From eq.~(7) of \cite{AgranatTamirEtAl25}, we conclude, as $n\rightarrow \infty$,
\begin{equation}
\label{eq:En1}
    \tilde{E}_{n,1} =[t^n] \mathcal{\tilde{E}}_1(t) \sim \frac{1}{\gamma^3 \sqrt{\pi}} n^{\frac{1}{2}} \rho^{-n}.
\end{equation}
The first values of $\tilde{E}_{n,1}$ are calculated in Section \ref{sec:exactulggt}. They appear in Table \ref{table:Etildng}. 

\subsection{Two galls}
\label{sec:twogalls}

For two galls, we have six cases (Figure \ref{fig:ggttwogalls}). If the root node is not part of a gall, then its two children are either (1) a general galled tree with two galls and a general galled tree with zero galls, or (2) two general galled trees with one gall each. If the root node is part of a gall, then there are four cases: (3) both sequences from the root to the reticulation node are nonempty and the second gall is descended from the reticulation node; (4) both sequences from the root to the reticulation node are nonempty and the second gall is on one of these paths; (5) one sequence from the root to the reticulation node is nonempty and contains the second gall, and the other sequence is empty; (6) one sequence from the root to the reticulation node is nonempty, the other is empty, and the second gall is descended from the reticulation node. In symbolic notation,
\begin{align}
\mathcal{\tilde{E}}_2 & = \{\circ\} \times 
\big[\underbrace{\mathcal{U} \times \mathcal{\tilde{E}}_2}_{1} 
\textrm{   }\dot{\cup}\textrm{   } \underbrace{\text{MSET}_2(\mathcal{\tilde{E}}_1)}_{2} 
\textrm{   }\dot{\cup}\textrm{   } \underbrace{\mathcal{\tilde{E}}_1 \times \text{MSET}_2 \big(\text{SEQ}^+(\mathcal{U}) \big)}_{3}
\textrm{   }\dot{\cup}\textrm{   } \underbrace{\mathcal{U} \times \big(\text{SEQ}(\mathcal{U}) \times \mathcal{\tilde{E}}_1 \times \text{SEQ}(\mathcal{U}) \big) \times \text{SEQ}^+(\mathcal{U})}_{4} \nonumber \\ &  
\quad \textrm{   }\dot{\cup}\textrm{   } \underbrace{\mathcal{U} \times \text{SEQ}(\mathcal{U}) \times \mathcal{\tilde{E}}_1 \times \text{SEQ}(\mathcal{U})}_{5}
\textrm{   }\dot{\cup}\textrm{   } \underbrace{\mathcal{\tilde{E}}_1 \times \text{SEQ}^+(\mathcal{U})}_{6}\big].
\label{eq:ulsmtwogalls}
\end{align}

\begin{figure}[tb]
    \centering
    \includegraphics[width=12cm]{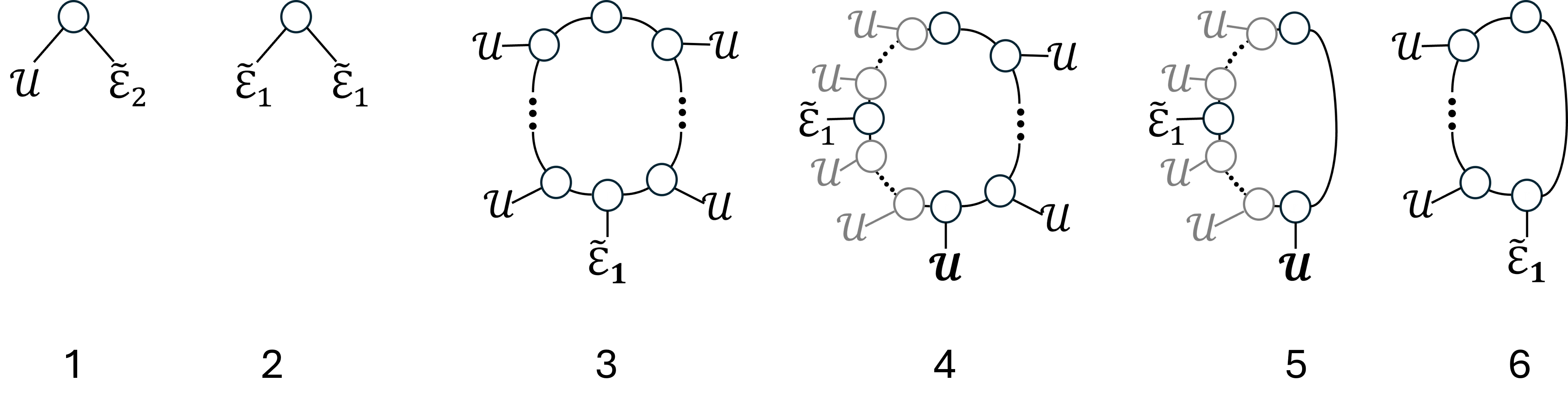}
    \vspace{-0.4cm}
    \caption{The six structures possible for a (non-plane) unlabeled general galled tree with exactly two galls.}
    \label{fig:ggttwogalls}
\end{figure}

Converting to a generating function $\mathcal{\tilde{E}}_2(t) = \sum_{n=0}^\infty \tilde{E}_{n,2} t^n$,
$$\mathcal{\tilde{E}}_2(t) = \mathcal{U}(t) \, \mathcal{\tilde{E}}_2(t) 
+ \frac{1}{2}[\mathcal{\tilde{E}}_1(t)^2 + \mathcal{\tilde{E}}_1(t^2)] 
+ \frac{\mathcal{\tilde{E}}_1(t)}{2} \bigg(\frac{\mathcal{U}(t)^2}{[1 - \mathcal{U}(t)]^2} + \frac{\mathcal{U}(t^2)}{1 - \mathcal{U}(t^2)}\bigg) 
+ \frac{\mathcal{U}(t)^2 \, \mathcal{\tilde{E}}_1(t)}{[1 - \mathcal{U}(t)]^3} 
+ \frac{\mathcal{U}(t) \, \mathcal{\tilde{E}}_1(t)}{[1 - \mathcal{U}(t)]^2} 
+ \frac{\mathcal{U}(t) \, \mathcal{\tilde{E}}_1(t)}{1 - \mathcal{U}(t)}.$$
Solving for $\mathcal{\tilde{E}}_2(t)$, we get:
\small
\begin{equation}
\mathcal{\tilde{E}}_2(t) = \frac{\mathcal{\tilde{E}}_1(t)^2}{2[1 - \mathcal{U}(t)]} 
+ \frac{\mathcal{\tilde{E}}_1(t^2)}{2[1 - \mathcal{U}(t)]} 
+ \frac{\mathcal{U}(t)^2 \, \mathcal{\tilde{E}}_1(t)}{2[1 - \mathcal{U}(t)]^3} 
+ \frac{\mathcal{U}(t^2) \, \mathcal{\tilde{E}}_1(t)}{2[1 - \mathcal{U}(t)] \,[1 - \mathcal{U}(t^2)]} 
+ \frac{\mathcal{U}(t)^2 \, \mathcal{\tilde{E}}_1(t)}{[1 - \mathcal{U}(t)]^4}
+ \frac{\mathcal{U}(t) \, \mathcal{\tilde{E}}_1(t)}{[1 - \mathcal{U}(t)]^3} + \frac{\mathcal{U}(t) \, \mathcal{\tilde{E}}_1(t)}{[1 - \mathcal{U}(t)]^2}.    
\end{equation}
\normalsize

In comparison with $\mathcal{E}_2(t) = \sum_{n=0}^\infty E_{n,2} t^n$, the generating function for the unlabeled time-consistent galled trees with two galls~\citep[eq.~(9)]{AgranatTamirEtAl25}, the generating function for the unlabeled general galled trees with two galls contains two terms that arise from cases 5 and 6 in Figure \ref{fig:ggttwogalls}. Neither is a term of highest degree; asymptotically, from eq.~(\ref{eq:utasymp}), as $t \rightarrow \rho^-$, $\mathcal{\tilde{E}}_2(t) \sim \mathcal{E}_2(t)$. 

From eq.~(10) of \cite{AgranatTamirEtAl25}, we conclude, as $n\rightarrow \infty$,
\begin{equation}
\label{eq:En2}
      \tilde{E}_{n,2} =[t^n] \mathcal{\tilde{E}}_2(t) \sim \frac{1}{3\gamma^7 \sqrt{\pi}} n^{\frac{5}{2}} \rho^{-n}.
\end{equation}
The first values of $\tilde{E}_{n,2}$ are calculated in Section \ref{sec:exactulggt}, and they appear in Table \ref{table:Etildng}. 

\subsection{Bivariate generating function}

As can be deduced from the cases of one and two galls, the generating function for unlabeled general galled trees with any specific number of galls $g \geq 1$, $\mathcal{\tilde{E}}_g(t)=\sum_{n=0}^\infty \tilde{E}_{n,g}t^n$, is dependent on the corresponding generating functions for galled trees with fewer galls.

Therefore, to find $\mathcal{\tilde{E}}_g(t)$, we first find the bivariate generating function, $\mathcal{\tilde{G}}(t,u)=\sum_{n=0}^\infty \sum_{m=0}^\infty \tilde{E}_{n,m}t^nu^m$, counting both the number of leaves and the number of galls. We obtain the generating function $\tilde{E}_g(t)$ for a fixed number of galls $g$ by noting $\cU(t)= \mathcal{\tilde{E}}_0(t) =  \mathcal{\tilde{G}}(t,0)$, and for $g \geq 1$,
\[
\mathcal{\tilde{E}}_g(t)=\rcp{g!}\bigg(\pdiff{}{u}{g}\tilde{\cG}\bigg)(t,0).
\]

In the symbolic method, general galled trees include the same components as  time-consistent galled trees \citep[eq.~(14)]{AgranatTamirEtAl25}, with the addition of a case that has a root gall and an empty sequence on one side of the reticulation node (Figure \ref{fig:ggtbv}). The symbolic method gives us 
\begin{equation}
   \mathcal{\tilde{G}} = \underbrace{\{ \square \}}_{1} \textrm{   } {\dot{\cup}} \textrm{   }
   \{ \circ \} \times \bigg[
   \underbrace{\textrm{MSET}_2(\mathcal{\tilde{G}})}_{2} \textrm{   } \dot{\cup} \textrm{   }
   \underbrace{\mu \times \mathcal{\tilde{G}} \times \textrm{MSET}_2\Big(\textrm{SEQ}^+(\mathcal{\tilde{G}})\Big)}_{3} \textrm{   } \dot{\cup} \textrm{   }
   \underbrace{\mu \times \mathcal{\tilde{G}} \times \textrm{SEQ}^+(\mathcal{\tilde{G}})}_{4}
   \bigg].
   \label{eq:bivariatesymbolic}
\end{equation}
The last component (4) is the one that appears in general galled trees but not in time-consistent galled trees. The generating function satisfies
\begin{equation}
      \mathcal{\tilde{G}}(t,u) = t
      +\frac{1}{2}[\mathcal{\tilde{G}}(t,u)^2+\mathcal{\tilde{G}}(t^2,u^2)]
      +\frac{u\mathcal{\tilde{G}}(t,u)}{2}\bigg[\bigg(\frac{\mathcal{\tilde{G}}(t,u)}{1-\mathcal{\tilde{G}}(t,u)}\bigg)^2+\frac{\mathcal{\tilde{G}}(t^2,u^2)}{1-\mathcal{\tilde{G}}(t^2,u^2)}\bigg] + \frac{u\mathcal{\tilde{G}}(t,u)^2}{1-\mathcal{\tilde{G}}(t,u)}. 
\label{eq:bivariategf}
\end{equation}

\begin{figure}[tb]
    \centering
    \includegraphics[width=6.8cm]{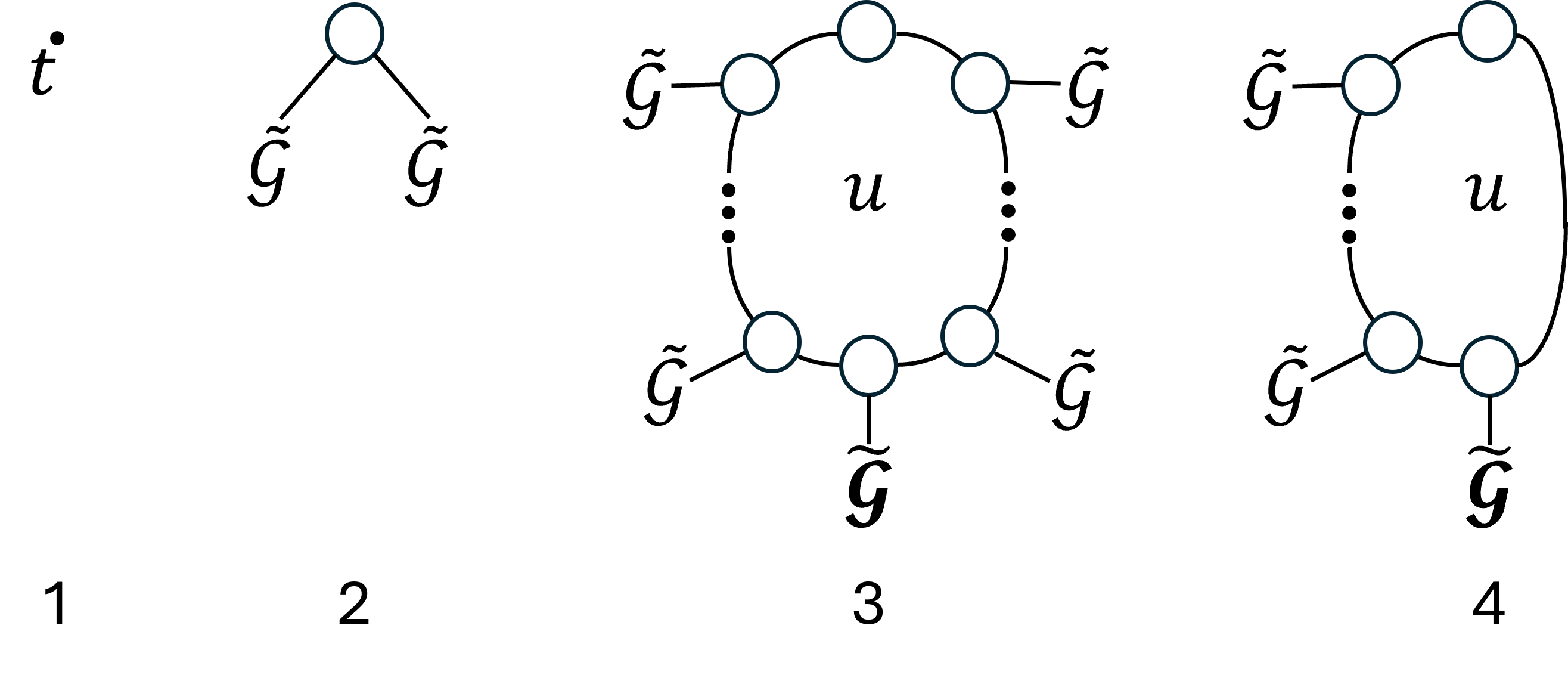}
    \vspace{-0.4cm}
    \caption{The four structures possible for a (non-plane) unlabeled general galled tree.}
    \label{fig:ggtbv}
\end{figure}

\subsection{Fixed number of galls} 
\label{sec:fixedulggt}

To find the univariate generating function for unlabeled general galled trees with $g$ galls, we must add the $g$th derivative at $u=0$ of the last term, ${u\mathcal{\tilde{G}}(t,u)^2}/[1-\mathcal{\tilde{G}}(t,u)]$, to the univariate generating function for unlabeled time-consistent galled trees~\citep[eqs.~(26)-(29)]{AgranatTamirEtAl25}. To take the required derivative, writing $\Du$ for $\frac{\partial}{\partial u}$, we use Leibniz's product rule: 
\begin{equation}
\Du^n(fh) = \sum_{k=0}^n {n \choose k} (\Du^k f)(\Du^{n-k} h).
\label{eq:leibniz}
\end{equation}
We also use Faà di Bruno's formula for derivatives of a composition:
\begin{equation}
\Du^m(f \circ h)=\sum_{k_1+2k_2+\ldots +mk_m=m} \frac{m!}{k_1! \, k_2! \, \cdots \, k_m!} \bigg[ \Dx^{k_1+k_2+\ldots
+k_m}\Big(f(x)\Big)\Big|_{x=h(u)} \bigg] \bigg[ \prod_{\ell=1}^m \bigg(\frac{\Du^{\ell} h}{\ell!}\bigg)^{k_\ell}\bigg].
\label{eq:fdb}
\end{equation}

We use Leibniz's product rule with $f(t,u)=u$, $h(t,u)=\mathcal{\tilde{G}}(t,u)^2 /[1-\mathcal{\tilde{G}}(t,u)]$, using $g$ in the role of $n$ (taking the $g$th derivative); at the value $u=0$, we notice that the only term that is not zero has $k=1$. We use Faà di Bruno's formula
with $f(x)=x^2/(1-x)$ and $h(u)=\tilde{\cG}(t,u)$, so that when $u=0$, $h({0})=\mathcal{U}(t)$:
\begin{align*} 
\rcp{g!} \Du^g\Bigg(u \frac{\tilde{\cG}(t,u)^2}{1-\tilde{\cG}(t,u)}\Bigg)\Bigg|_{u=0}&=
\rcp{(g-1)!} \Du^{g-1}\Bigg(\frac{\tilde{\cG}(t,u)^2}{1-\tilde{\cG}(t,u)}\Bigg)\Bigg|_{u=0} \nonumber \\
&=\sum_{k_1+2k_2+\ldots \atop +(g-1)k_{g-1}=g-1} \rcp{k_1!\, k_2!\cdots k_{g-1}!} 
\Dx^{k_1+k_2+\ldots+k_{g-1}}\bigg(\frac{x^2}{1-x}\bigg)\bigg|_{x=\cU(t)} \prod_{\ell=1}^{g-1} \tilde{\cE}_\ell(t)^{k_\ell} 
\end{align*} 

We then insert 
\begin{equation*} 
\rcp{k!}\Dx^k \bigg(\frac{x^2}{1-x}\bigg) = 
\begin{cases}
\frac{2x-x^2}{(1-x)^2} = \frac{1}{(1-x)^2}-1 & \text{ if } k=1,\\
\frac{1}{(1-x)^{k+1}} & \text{ if } k \geq 2.
\end{cases}
\end{equation*} 
The extra term in $\mathcal{\tilde{E}}_g(t)$ is 
\begin{equation}
    \sum_{k_1+2k_2+\ldots +(g-1)k_{g-1}=g-1} \binom{k_1+k_2+ \ldots +k_{g-1}}{k_1,k_2,\ldots,k_{g-1}} 
\Bigg(\frac{1}{\big[1-\cU(t)\big]^{(\sum_{i=1}^{g-1}k_i)+1}}-\delta_{1,\,\sum_{i=1}^{g-1}k_i}\Bigg)
 \prod_{m=1}^{g-1} \tilde{\cE}_m(t)^{k_m}.  \nonumber \\
\label{eq:extratermggt}
\end{equation}

Adding the extra term to the expression in eqs.~(26)-(29) of \cite{AgranatTamirEtAl25}, we get 
\small
\begin{align}
\tilde{\cE}_g(t) &= \frac{1}{1-\cU(t)}\bigg[\Big( \rcp2\sum_{\ell=1}^{g-1} \tilde{\cE}_\ell(t) \, \tilde{\cE}_{g-\ell}(t) \Big) +\delta_{0,g \textrm{ mod }2}\rcp2\tilde{\cE}_{\frac g2}(t^2) \nonumber \\
& \quad +\rcp2\sum_{k_1+2k_2+\ldots +(g-1)k_{g-1}=g-1} \binom{k_1+k_2+ \ldots +k_{g-1}}{k_1,k_2,\ldots,k_{g-1}} 
\Bigg(\frac{3\cU(t)+(\sum_{i=1}^{g-1}k_i)-2}{\big[1-\cU(t)\big]^{(\sum_{i=1}^{g-1}k_i)+2}}+\delta_{1,\,\sum_{i=1}^{g-1}k_i}\Bigg) 
\prod_{m=1}^{g-1} \tilde{\cE}_m(t)^{k_m} \nonumber \\ 
& \quad +\rcp2\sum_{b=1}^{\big\lfloor\frac{ g-1}{2} \big\rfloor} \tilde{\cE}_{(g-2b)-1}(t) \sum_{2r_2+4r_4+\ldots +(2b)r_{2b}=2b} \binom{r_1+r_2+\ldots +r_{2b}}{r_1,r_2,\ldots,r_{2b}} 
\rcp{\big[1-\cU(t^2)\big]^{(\sum_{i=1}^{b}r_{2i})+1}}  
\prod_{m=1}^{ b}
\tilde{\cE}_m(t^2)^{r_{2m}}  \nonumber \\ 
& \quad +\rcp2 \tilde{\cE}_{g-1}(t)\frac{\cU(t^2)}{1-\cU(t^2)} \nonumber \\
& \quad +\sum_{k_1+2k_2+\ldots +(g-1)k_{g-1}=g-1} \binom{k_1+k_2+ \ldots +k_{g-1}}{k_1,k_2,\ldots,k_{g-1}} 
\Bigg(\frac{1}{\big[1-\cU(t)\big]^{(\sum_{i=1}^{g-1}k_i)+1}}-\delta_{1,\,\sum_{i=1}^{g-1}k_i}\Bigg)
\prod_{m=1}^{g-1} \tilde{\cE}_m(t)^{k_m}\bigg], \nonumber
\end{align}
\normalsize
or 
\small
\begin{align}
\tilde{\cE}_g(t) &= \frac{1}{1-\cU(t)}\bigg[\Big( \rcp2\sum_{\ell=1}^{g-1} \tilde{\cE}_\ell(t) \, \tilde{\cE}_{g-\ell}(t) \Big) +\delta_{0,g \textrm{ mod }2}\rcp2\tilde{\cE}_{\frac g2}(t^2) 
\nonumber  \\
& \quad +\sum_{k_1+2k_2+\ldots +(g-1)k_{g-1}=g-1} \binom{k_1+k_2+ \ldots +k_{g-1}}{k_1,k_2,\ldots,k_{g-1}} \nonumber \\
& \qquad \times \Bigg(\frac{3\cU(t)+(\sum_{i=1}^{g-1}k_i)-2}{2\big[1-\cU(t)\big]^{(\sum_{i=1}^{g-1}k_i)+2}}
+\frac{1}{\big[1-\cU(t)\big]^{(\sum_{i=1}^{g-1}k_i)+1}}
-\rcp{2}\delta_{1,\,\sum_{i=1}^{g-1}k_i}\Bigg) \prod_{m=1}^{g-1} \tilde{\cE}_m(t)^{k_m} \nonumber \\
& \quad +\rcp2\sum_{b=1}^{\big\lfloor\frac{ g-1}{2} \big\rfloor} \tilde{\cE}_{(g-2b)-1}(t) \sum_{2r_2+4r_4+\ldots +(2b)r_{2b}=2b} \binom{r_1+r_2+\ldots +r_{2b}}{r_1,r_2,\ldots,r_{2b}} 
\rcp{\big[1-\cU(t^2)\big]^{(\sum_{i=1}^{b}r_{2i})+1}} \prod_{m=1}^{ b}
\tilde{\cE}_m(t^2)^{r_{2m}} \label{eq:eg3a} \nonumber \\ 
& \quad +\rcp2 \tilde{\cE}_{g-1}(t)\frac{\cU(t^2)}{1-\cU(t^2)}.
\end{align}
\normalsize

Writing $\ell=\sum_{i=1}^{g-1}k_i$ and using eq.~(\ref{eq:utasymp}), we notice that $[3\cU(t)+\ell-2]/\big[2\big(1-\cU(t)\big)^{\ell+2}\big]\sim (\ell+1)/[{2}\gamma^{\ell+2}(1-t/\rho)^{\ell/2+1}]$ and $1/[1-\cU(t)]^{\ell+1} \sim 1 /[\gamma^{\ell+1}(1-t/\rho)^{{(\ell+1)/2}}]$. The term $\delta_{1,\ell}$ is a constant with respect to $t$ and $\ell>0$. In addition, as shown in Section \ref{sec:ggt1gall}, $\mathcal{\tilde{E}}_1(t) \sim \mathcal{E}_1(t)$, and inductively we can conclude that for all $m$, $0 \leq m < g$, $\mathcal{\tilde{E}}_m(t) \sim \mathcal{E}_m(t)$. Therefore, the term unique to the generating function of unlabeled general galled trees does not contribute to the asymptotics of the function. Hence, the asymptotic analysis is identical to that of unlabeled time-consistent galled trees \citep[Theorem 10]{AgranatTamirEtAl24b}, so that as $n \rightarrow \infty$, 
\begin{equation}
\label{eq:Eng}
\tilde{E}_{n,g} =[t^n] \mathcal{\tilde{E}}_g(t) \sim \frac{2^{2g-1}}{(2g)! \, \gamma^{4g-1} \sqrt{\pi} } n^{2g-\frac{3}{2} } \rho^{-n}. 
\end{equation}
Values of $\tilde{E}_{n,g}$ for small $n$ and $g$ are calculated in Section \ref{sec:exactulggt}, and they appear in Table \ref{table:Etildng}. 

\subsection{Arbitrary number of galls} 
\label{sec:allulggt}

Fuchs \& Gittenberger \cite{FuchsAndGittenberger25} analyzed unlabeled general galled trees with all possible numbers of galls; we derive the generating function for completion. There are four possible structures (Figure \ref{fig:ggtbv}): (1) the tree is a single node; (2) it has no root gall and two galled subtrees; (3) it has a root gall in which the sequences leading to the reticulation node are both nonempty, and from all nodes in the root gall, except the root, galled trees descend; (4) it has a root gall in which one of the sequences leading to the reticulation node is empty, and from all nodes in the root gall, except the root, galled trees descend. In symbols,
\begin{equation}
    \label{eq:atul}
    \mathcal{\tilde{A}} = \underbrace{\{\square\}}_{1} 
\textrm{   }\dot{\cup}\textrm{   } \{\circ\} \, \times \big[\underbrace{\text{MSET}_2(\mathcal{\tilde{A}})}_{2} 
\textrm{   } \dot{\cup} \textrm{   } \underbrace{\mathcal{\tilde{A}}\times\text{MSET}_2 \big(\text{SEQ}^+(\mathcal{\tilde{A}}) \big)}_{3}   
\textrm{   } \dot{\cup} \textrm{   } \underbrace{\mathcal{\tilde{A}}\times\text{SEQ}^+(\mathcal{\tilde{A}})}_{4}  \big].
\end{equation}

Converting to a generating function $\mathcal{\tilde{A}} = \sum_{n=0}^\infty \tilde{A}_n t^n$,
\begin{equation}
\mathcal{\tilde{A}}(t) = t + \frac{1}{2}[\mathcal{\tilde{A}}(t)^2 + \mathcal{\tilde{A}}(t^2)] + \frac{\mathcal{\tilde{A}}(t)}{2}\bigg[\bigg(\frac{\mathcal{\tilde{A}}(t)}{1 - \mathcal{\tilde{A}}(t)}\bigg)^2 + \frac{\mathcal{\tilde{A}}(t^2)}{1 - \mathcal{\tilde{A}}(t^2)}\bigg] + \frac{\mathcal{\tilde{A}}(t)^2}{1 - \mathcal{\tilde{A}}(t)}.
\label{eq:atulgf}
\end{equation}
Eq.~(\ref{eq:atulgf}) matches eq.~(28) of \cite{FuchsAndGittenberger25}, which asymptotically showed
\begin{equation}
\tilde{A}_n \sim \frac{0.19659...}{2\sqrt{\pi}} n^{-\frac{3}{2}} (0.11647...)^{-n}.
\label{eq:fg}
\end{equation}

The first values of $\tilde{A}_n$ are calculated in Section \ref{sec:exactulggt} as the sum $\sum_{g=0}^{n-1}\tilde{E}_{n,g}$. They appear in Table {\ref{table:Etildng}}.

\section{Leaf-labeled general galled trees} 
\label{sec:llggt}

The derivations of generating functions for leaf-labeled general galled trees are similar to those of unlabeled general galled trees. The general galled trees have an additional component, compared to time-consistent galled trees, which does not affect the asymptotics.

\subsection{No galls}
\label{sec:nogalls}

Leaf-labeled general galled trees with no galls are the same as leaf-labeled time-consistent galled trees with no galls. Both are leaf-labeled rooted binary trees. The exponential generating function satisfies~\citep[p.~129]{FlajoletAndSedgewick09}:
\begin{equation}
    \mathfrak{U}(t) = \sum_{n=0}^\infty \frac{u_n}{n!}t^n = t + \frac{1}{2}\mathfrak{U}(t)^2,
    \label{eq:utll}
\end{equation}
where the $u_n$ count leaf-labeled trees with $n$ leaves. The $\{u_n\}_{n\geq2}$ satisfy $u_n=(2n-3)!!$ and $u_n=1$ \citep{EdwardsAndCSforza64, Felsenstein78} The factor of $\frac{1}{2}$ in eq.~(\ref{eq:utll}) arises because the trees are non-plane. The exponential generating function is
\begin{equation} \label{eq:U-labeled}
\mathfrak{U}(t) = 1-\sqrt{1-2t},
\end{equation}
so that $\frac{1}{2}$ is the radius of convergence. The equation has the same form as the corresponding generating function for unlabeled rooted binary trees as $t \rightarrow \rho^-$ (eq.~(\ref{eq:utasymp})), substituting $\rho=\frac{1}{2}$ and $\gamma=1$.

To obtain the asymptotic approximation to $u_n$, as $n \rightarrow \infty$, we apply Figure VI.5 of \cite{FlajoletAndSedgewick09}, producing
\begin{equation}
u_n \sim \frac{1}{2\sqrt{\pi}} n^{-\frac{3}{2}} \Big(\frac{1}{2}\Big)^{-n}n!\sim \frac{\sqrt{2}}{2}n^{n-1}\Big(\frac{2}{e}\Big)^n,
\end{equation}
where the last approximation follows by Stirling's approximation to $n!$. The $u_n$ follow OEIS A001147. The first values of $u_n$ appear in Table \ref{table:etildng}.

\subsection{One gall} 
\label{sec:llggt1g}

The structure of leaf-labeled general galled trees follows that of unlabeled general galled trees. From Figure \ref{fig:ggtonegall}, in the notation for the symbolic method for labeled classes \cite[p.~148]{FlajoletAndSedgewick09}:
\begin{equation}
\mathfrak{\tilde{E}}_1 =  \textrm{   }  \textrm{   } 
    \{\circ\} \times    \bigg[\underbrace{\mathfrak{U}\star\mathfrak{\tilde{E}}_1}_{1} \textrm{   } \dot{\cup} \textrm{   } \underbrace{\mathfrak{\tilde{U}}\star \textrm{SET}_2\big(\textrm{SEQ}^+(\mathfrak{\tilde{U}})\big)}_{2}\textrm{   } \dot{\cup}\textrm{   }
   \underbrace{\mathfrak{\tilde{U}} \star \textrm{SEQ}^+(\mathfrak{\tilde{U}})}_{3}
    \bigg].
\end{equation}
The exponential generating function $\mathfrak{\tilde{E}}_1(t) = \sum_{n=0}^\infty (\tilde{e}_{n,1} /n! ) t^n$ is
\begin{align}
    \mathfrak{\tilde{E}}_1(t) &= \mathfrak{\tilde{U}}(t) \, \mathfrak{\tilde{E}}_1(t) + \frac{\mathfrak{\tilde{U}}(t)}{2} \bigg ( \frac{\mathfrak{\tilde{U}}(t)}{1-\mathfrak{\tilde{U}}(t)} \bigg )^2 + \frac{\mathfrak{\tilde{U}}(t)^2}{1-\mathfrak{\tilde{U}}(t)} \nonumber \\
     &= \frac{\mathfrak{\tilde{U}}(t)^3}{2[1-\mathfrak{\tilde{U}}(t)]^3} + \frac{\mathfrak{\tilde{U}}(t)^2}{[1-\mathfrak{\tilde{U}}(t)]^2}.
     \label{eq:llggtonegall}
\end{align}

Comparing eq.~(\ref{eq:llggtonegall}) to $\mathfrak{E}_1(t)=\sum_{n=0}^{\infty} (e_{n,1}/n! ) t^n$, the exponential generating function of leaf-labeled time-consistent galled trees with one gall (eq.~(37) in \cite{AgranatTamirEtAl25}), eq.~(\ref{eq:llggtonegall}) has the additional term ${\mathfrak{\tilde{U}}(t)^2}/{[1-\mathfrak{\tilde{U}}(t)]^2}$. As in the unlabeled case (Section \ref{sec:ggt1gall}), the extra term is not of highest order and hence, as $t \rightarrow \frac{1}{2}^-$, $\mathfrak{\tilde{E}}_1(t) \sim \mathfrak{E}_1(t)$. 

From eq.~(40) in \cite{AgranatTamirEtAl25}, as $n\rightarrow \infty$,
\begin{equation}
    \tilde{e}_{n,1}  = [t^n]\mathfrak{\tilde{E}}_1(t) \, n! \sim  \frac{1}{\sqrt{\pi}} n^{\frac{1}{2}} \Big(\frac{1}{2}\Big)^{-n}n!
    \sim \sqrt{2}n^{n+1}\Big(\frac{2}{e}\Big)^n. \label{eq:en1asy}
\end{equation}
The first values of $\tilde{e}_{n,1}$ are calculated in Section \ref{sec:exactllggt}. They appear in Table \ref{table:etildng}.

\subsection{Two galls} 
\label{sec:llggt2g}

As in the case of one gall, the structure of a leaf-labeled general galled tree with two galls follows that of a corresponding unlabeled general galled tree. From Figure \ref{fig:ggttwogalls}, we get in the leaf-labeled case:
\begin{align}
\mathcal{\tilde{\mathfrak{E}}}_2 & = \{\circ\} \times 
\big[\underbrace{\mathfrak{U} \star \mathfrak{\tilde{E}}_2}_{1} 
\textrm{   }\dot{\cup}\textrm{   }
\underbrace{\text{SET}_2(\mathfrak{\tilde{E}}_1)}_{2} 
\textrm{   }\dot{\cup}\textrm{   }
\underbrace{\mathfrak{U} \star \big(\text{SEQ}(\mathfrak{U}) \star \mathfrak{\tilde{E}}_1 \star \text{SEQ}(\mathfrak{U})\big) \star \text{SEQ}^+(\mathfrak{U})}_{3}
\textrm{   }\dot{\cup} \textrm{   }
\underbrace{\mathfrak{\tilde{E}}_1 \star \text{SET}_2 \big(\text{SEQ}^+(\mathfrak{U}) \big)}_{4} \nonumber \\ &  
\quad \textrm{   }\dot{\cup} \textrm{   }\underbrace{\mathfrak{U} \star \text{SEQ}(\mathfrak{U}) \star \mathfrak{\tilde{E}}_1 \star \text{SEQ}(\mathfrak{U})}_{5}
\textrm{   }\dot{\cup} \textrm{   }
\underbrace{\mathfrak{\tilde{E}}_1 \star \text{SEQ}^+(\mathfrak{U})}_{6}\big].
\label{eq:llsmtwogalls}
\end{align}

Converting to an exponential generating function $\tilde{\mathfrak{E}}_2(t) = \sum_{n=0}^\infty (\tilde{e}_{n,2}/n!)t^n$,
$$\mathfrak{\tilde{E}}_2(t) = \mathfrak{U}(t) \, \mathfrak{\tilde{E}}_2(t) 
+ \frac{1}{2}\mathfrak{\tilde{E}}_1(t)^2 
+ \frac{\mathfrak{U}(t)^2 \, \mathfrak{\tilde{E}}_1(t)}{[1 - \mathfrak{U}(t)]^3} 
+ \frac{\mathfrak{\tilde{E}}_1(t)}{2}\bigg(\frac{\mathfrak{U}(t)}{1 - \mathfrak{U}(t)}\bigg)^2 
+ \frac{\mathfrak{U}(t) \, \mathfrak{\tilde{E}}_1(t)}{[1 - \mathfrak{U}(t)]^2} 
+ \frac{\mathfrak{U}(t) \, \mathfrak{\tilde{E}}_1(t)}{1 - \mathfrak{U}(t)}.$$
Solving for $\mathfrak{\tilde{E}}_2(t)$, we have
\begin{equation}
\mathfrak{\tilde{E}}_2(t) = \frac{\mathfrak{\tilde{E}}_1(t)^2}{2[1 - \mathfrak{U}(t)]} 
+ \frac{\mathfrak{U}(t)^2 \, \mathfrak{\tilde{E}}_1(t)}{[1 - \mathfrak{U}(t)]^4}
+ \frac{\mathfrak{U}(t)^2 \, \mathfrak{\tilde{E}}_1(t)}{2[1 - \mathfrak{U}(t)]^3} 
+ \frac{\mathfrak{U}(t) \, \mathfrak{\tilde{E}}_1(t)}{[1 - \mathfrak{U}(t)]^3}
+ \frac{\mathfrak{U}(t) \, \mathfrak{\tilde{E}}_1(t)}{[1 - \mathcal{U}(t)]^2}.    
\end{equation}

In comparison with $\mathfrak{E}_2(t) = \sum_{n=0}^\infty (e_{n,2}/n!) t^n$, the exponential generating function for the leaf-labeled time-consistent galled trees with two galls~\citep[eq.~(43)]{AgranatTamirEtAl25}, the generating function for the leaf-labeled general galled trees with two galls contains two terms that arise from cases 5 and 6 in Figure \ref{fig:ggttwogalls}. In the same manner as in the unlabeled case (Section \ref{sec:twogalls}), neither is a term of highest degree. Therefore, asymptotically, as $t \rightarrow \frac{1}{2}^-$, $\mathfrak{\tilde{E}}_2(t) \sim \mathfrak{E}_2(t)$. 

From eq.~(44) of \cite{AgranatTamirEtAl25}, as $n\rightarrow \infty$,
\begin{equation}
    \tilde{e}_{n,2} = [t^n]\mathfrak{\tilde{E}_2(t)} \, n! \sim  \frac{1}{3\sqrt{\pi}} n^{\frac{5}{2}} \Big(\frac{1}{2}\Big)^{-n}n! \sim \frac{\sqrt{2}}{3}n^{n+3}\Big(\frac{2}{e}\Big)^n. \label{eq:en2asy}
\end{equation}
The first values of $\tilde{e}_{n,2}$ are calculated in Section \ref{sec:exactllggt} and they appear in Table \ref{table:etildng}.

\subsection{Bivariate exponential generating function}

We follow the same procedure we used in the unlabeled case and arrive at the exponential generating function that counts leaf-labeled general galled trees with specific numbers of galls. We first find the bivariate exponential generating function $\mathfrak{\tilde{G}}(t,u) = \sum_{n=0}^\infty \sum_{m=0}^\infty (\tilde{e}_{n,g}/n! ) t^n u^m$, counting both leaves and galls. From the structure in Figure \ref{fig:ggtbv}, the symbolic method yields
\begin{equation}
\mathfrak{\tilde{G}} = \underbrace{\{\square\}}_{1} 
\textrm{   } \dot{\cup} \textrm{   } 
    \{\circ\} \times  \bigg[\underbrace{\textrm{SET}_2(\mathfrak{\tilde{G}})}_{2} 
    \textrm{   } \dot{\cup} \textrm{   } \underbrace{\mu\times\mathfrak{\tilde{G}}\star \textrm{SET}_2\big(\textrm{SEQ}^+(\mathfrak{\tilde{G}})\big)}_{3}
    \textrm{   } \dot{\cup} \textrm{   }
   \underbrace{\mu \times \mathfrak{\tilde{G}} \star \textrm{SEQ}^+(\mathfrak{\tilde{G}})}_{4}
    \bigg].
\end{equation}
The additional term (component 4) compared to leaf-labeled time-consistent galled trees (eq.~(48) of \cite{AgranatTamirEtAl25}) is similar to that in the unlabeled case (eq.~(\ref{eq:bivariatesymbolic})).
We therefore have 
\begin{equation} \label{eq:llggtbvf}
      \mathfrak{\tilde{G}}(t,u) = t
      +\frac{1}{2}\mathfrak{\tilde{G}}(t,u)^2
      +\frac{{u}\mathfrak{\tilde{G}}(t,u)}{2}\bigg(\frac{\mathfrak{\tilde{G}}(t,u)}{1-\mathfrak{\tilde{G}}(t,u)}\bigg)^2 
      + \frac{u\mathfrak{\tilde{G}}(t,u)^2}{1-\mathfrak{\tilde{G}}(t,u)}.
\end{equation}

\subsection{Fixed number of galls} 
\label{sec:fixedllggt}

The exponential generating function $\mathfrak{\tilde{E}}_g(t)=\sum_{n=0}^{\infty} (\tilde{e}_{n,g}/n!)t^n$ that counts leaf-labeled general galled trees is the sum of the generating function of leaf-labeled time consistent galled trees $\mathfrak{E}_g(t)$ and the last term in eq.~(\ref{eq:llggtbvf}), ${u\mathfrak{\tilde{G}}(t,u)^2}/{[1-\mathfrak{\tilde{G}}(t,u)]}$. We have already calculated the $g$th derivative of this extra term at $u=0$ (Section \ref{sec:fixedulggt}); hence, combining that calculation (eq.~(\ref{eq:extratermggt})) with eq.~(51) of \cite{AgranatTamirEtAl25}, we have
\small
\begin{align}
  \mathfrak{\tilde{E}}_g(t) &= \frac{1}{1-\mathfrak{U}(t)}\bigg[\frac{1}{2}\sum_{\ell=1}^{g-1} \mathfrak{\tilde{E}}_\ell(t) \, \mathfrak{\tilde{E}}_{g-\ell}(t) \nonumber\\  
  & 
\quad + \frac{1}{2}\sum_{k_1+2k_2+\ldots +(g-1)k_{g-1}=g-1} \binom{k_1+k_2+\ldots +k_{g-1}}{k_1,k_2,\ldots,k_{g-1}} \bigg(\frac{3\mathfrak{U}(t)+(\sum_{i=1}^{g-1}k_i)-2}{[1-\mathfrak{U}(t)]^{(\sum_{i=1}^{g-1}k_i)+2}}+\delta_{1,\,\sum_{i=1}^{g-1}k_i}\Bigg)
\prod_{\ell=1}^{g-1} \mathfrak{\tilde{E}}_\ell(t)^{k_\ell} \nonumber \\
&  \quad + \sum_{k_1+2k_2+\ldots +(g-1)k_{g-1}=g-1} \binom{k_1+k_2+ \ldots +k_{g-1}}{k_1,k_2,\ldots,k_{g-1}} 
\Bigg(\frac{1}{\big[1-\mathfrak{U}(t)\big]^{(\sum_{i=1}^{g-1}k_i)+1}}-\delta_{1,\,\sum_{i=1}^{g-1}k_i}\Bigg)
\prod_{m=1}^{g-1} \mathfrak{\tilde{E}}_m(t)^{k_m}  \bigg], \nonumber
\end{align}
\normalsize
or 
\begin{align}
    \mathfrak{\tilde{E}}_g(t) &= \frac{1}{1-\mathfrak{U}(t)}\bigg[\bigg( \rcp2\sum_{\ell=1}^{g-1} \mathfrak{\tilde{E}}_\ell(t) \, \mathfrak{\tilde{E}}_{g-\ell}(t) \bigg) \nonumber \\
& \quad +\sum_{k_1+2k_2+\ldots +(g-1)k_{g-1}=g-1} \binom{k_1+k_2+ \ldots +k_{g-1}}{k_1,k_2,\ldots,k_{g-1}} \nonumber \\
& \quad \times \Bigg(\frac{3\cU(t)+(\sum_{i=1}^{g-1}k_i)-2}{2\big[1-\cU(t)\big]^{(\sum_{i=1}^{g-1}k_i)+2}}
+\frac{1}{\big[1-\cU(t)\big]^{(\sum_{i=1}^{g-1}k_i)+1}}
-\rcp{2}\delta_{1,\,\sum_{i=1}^{g-1}k_i}\Bigg) \prod_{m=1}^{g-1} \mathfrak{\tilde{E}}_m(t)^{k_m}. 
\end{align}

As in the unlabeled case (Section \ref{sec:fixedulggt}), we inductively obtain that for all $m$, $0 \leq m < g$,  $\mathfrak{\tilde{E}}_m (t) \sim \mathfrak{E}_m (t)$; the extra term in the exponential generating function for leaf-labeled general galled trees compared to leaf-labeled time-consistent galled trees does not contribute to the asymptotics. The asymptotic analysis of $\mathfrak{\tilde{E}}_g(t)$ is identical to that of $\mathfrak{E}_g(t)$ \citep[eqs.~(52)-(53)]{AgranatTamirEtAl25}: as $n\rightarrow \infty$, 
\begin{align}
\label{eq:eng}
     \mathfrak{E}_g(t) &\sim \frac{(4g-3)!!}{(2g)! \, (1-2t)^{2g-\frac{1}{2}}}, \\
    \tilde{e}_{n,g} &\sim \frac{2^{2g-1}}{(2g)!\sqrt{\pi}}n^{2g-\frac{3}{2}}\Big(\frac{1}{2}\Big)^{-n}n! \sim \frac{2^{2g-1}\sqrt{2}}{(2g)!}\Big(\frac{2}{e}\Big)^n n^{n+2g-1}. \label{eq:identical}
\end{align}
Values of $\tilde{e}_{n,g}$ for small $n$ and $g$ are calculated in Section \ref{sec:exactllggt}. They appear in Table \ref{table:etildng}.

\subsection{Arbitrary number of galls} 
\label{sec:allllggt}

As in the case of unlabeled general galled trees (Section \ref{sec:allulggt}), the analysis of leaf-labeled general galled trees with any possible number of galls has been considered by  \cite{FuchsAndGittenberger25} and earlier by \cite{BouvelEtAl20}. The tree structure in the leaf-labeled case follows the unlabeled case (Figure \ref{fig:ggtbv}), and the symbolic method gives
\begin{equation}
    \mathfrak{\tilde{A}} = \underbrace{\{\square\}}_{1} 
\textrm{   }\dot{\cup}\textrm{   } \{\circ\} \times \big[\underbrace{\text{SET}_2(\mathfrak{\tilde{A}})}_{2} 
\textrm{   }\dot{\cup}\textrm{   } \underbrace{\mathfrak{\tilde{A}}\star\text{SET}_2 \big(\text{SEQ}^+(\mathfrak{\tilde{A}}) \big)}_{3}   
\textrm{   }\dot{\cup}\textrm{   } \underbrace{\mathfrak{\tilde{A}}\star\text{SEQ}^+(\mathfrak{\tilde{A}})}_{4}  \big].
\label{eq:atll}
\end{equation}
Converting to an exponential generating function $\mathfrak{\tilde{A}}(t) = \sum_{n=0}^\infty (\tilde{a}_n/n!)t^n$, we have
\begin{equation}
\mathfrak{\tilde{A}}(t) = t + \frac{1}{2}\mathfrak{\tilde{A}}(t)^2 + \frac{\mathfrak{\tilde{A}}(t)}{2} \bigg(\frac{\mathfrak{\tilde{A}}(t)}{1 - \mathfrak{\tilde{A}}(t)}\bigg)^2 + \frac{\mathfrak{\tilde{A}}(t)^2}{1 - \mathfrak{\tilde{A}}(t)}.
\label{eq:atllgf}
\end{equation}
Eq.~(\ref{eq:atllgf}) is the same as eq.~(2) in \cite{FuchsAndGittenberger25}.

From Proposition 6 of \cite{BouvelEtAl20} and eq.~(3) of \cite{FuchsAndGittenberger25}, as $n\rightarrow \infty$, 
\begin{equation}
    \label{eq:sqrt17}
\tilde{a}_n \sim \frac{(17-\sqrt{17})\sqrt{2}}{136}n^{n-1} \Big(\frac{8}{e}\Big)^n.
\end{equation}
The first values of $\tilde{a}_n$ are calculated in Section \ref{sec:exactllggt} as the sum $\sum_{g=0}^{n-1}\tilde{e}_{n,g}$. They appear in Table {\ref{table:etildng}}.

\section{Unlabeled simplex time-consistent galled trees} \label{sec:ul1tcgt}

We study the same series of problems for the unlabeled simplex time-consistent galled trees that we examined for unlabeled general galled trees in Section \ref{sec:ulggt}. Let $\mathcal{E}_g^1(t)=\sum_{n=0}^\infty E^1_{n,g}t^n$ denote the generating function counting the number of unlabeled simplex galled trees with exactly $g$ galls, with $n$ leaves ($E^1_{n,g}$). As in the case of general galled trees, we examine the special cases of $g=0$, 1, and 2. We derive a bivariate generating function $\mathcal{G}^1(t,u) = \sum_{n=0}^\infty \sum_{m=0}^\infty E^1_{n,m} t^n u^m$, from which we analyze the growth of sequences with any fixed number of galls $g$. We also derive a generating function 
$\mathcal{A}^1(t) = \sum_{n=0}^\infty A_n^1 t^n$, counting unlabeled simplex time-consistent galled trees with $n$ leaves ($A_n^1$), summing over all possible numbers of galls.

\subsection{No galls}
\label{sec:1compnogalls}

The unlabeled simplex time-consistent galled trees with no galls, are, like the unlabeled time-consistent galled trees with no galls and the unlabeled general galled trees with no galls, simply the unlabeled rooted binary trees studied in Section \ref{sec:ggt0galls}. 

\subsection{One gall} 
\label{sec:1tcgt1g}

If an unlabeled simplex time-consistent galled tree contains exactly one gall, then as in Section \ref{sec:ggt1gall}, we split into cases depending on whether or not the root node is part of the gall (Figure \ref{fig:1tcgtonegall}). If the root is not part of the gall, then it has two subtrees, a tree with no galls and a simplex time-consistent galled tree with one gall (component 1 in eq.~(\ref{eq:e11tsm}) and Figure \ref{fig:1tcgtonegall}). If the root is part of the gall (component 2), then two non-empty sequences of trees with no galls lead from the root to the reticulation node, from which a single leaf descends, due to the simplex condition. The sequences are non-empty as a result of the time-consistency condition.

\begin{figure}[tb]
    \centering
    \includegraphics[width=4.5cm]{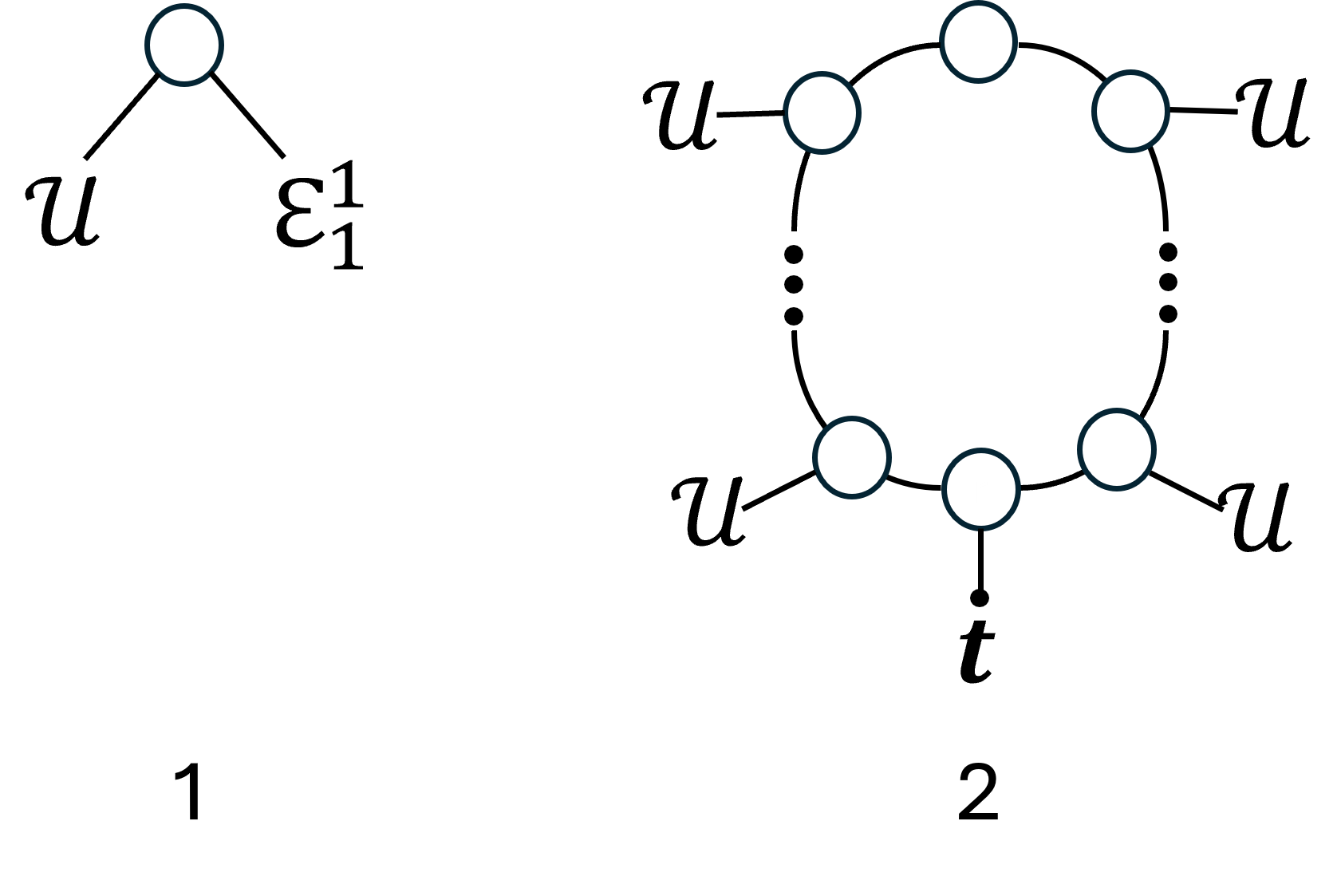}
    \vspace{-0.3cm}
    \caption{The two structures possible for a (non-plane) unlabeled simplex time-consistent galled tree with exactly one gall.}
    \label{fig:1tcgtonegall}
\end{figure}

In symbolic notation, we get 
\begin{equation}
\mathcal{E}_1^1 = \{\circ\} \times \Big[\underbrace{\mathcal{U} \times \mathcal{E}_1^1}_{1} 
\textrm{  }\dot{\cup}\textrm{  }
\underbrace{\{\square\} \times \text{MSET}_2 \big(\text{SEQ}^+(\mathcal{U})\big)}_{2} \Big].
    \label{eq:e11tsm}
\end{equation}
Converting to a generating function $\mathcal{E}_1^1(t) = \sum_{n=0}^\infty E_{n,1}^1 t^n$  yields 
$$\mathcal{E}_1^1(t) = \mathcal{U}(t) \, \mathcal{E}_1^1(t) + \frac{t}{2} \bigg[ \bigg( \frac{\mathcal{U}(t)}{1 - \mathcal{U}(t)}\bigg)^2 + \frac{\mathcal{U}(t^2)}{1 - \mathcal{U}(t^2)}\bigg].$$ 
Isolating $\mathcal{E}_1^1(t)$ gives 
\begin{equation}
\mathcal{E}_1^1(t) = \frac{t \, \mathcal{U}(t)^2}{2[1 - \mathcal{U}(t)]^3} + \frac{t \, \mathcal{U}(t^2)}{2[1 - \mathcal{U}(t)] \, [1 - \mathcal{U}(t^2)]}.
\end{equation}

The $t \, \mathcal{U}(t)^2 / \big( 2[1-\mathcal{U}(t)]^3\big)$ term dominates the asymptotic expression for the generating function $\mathcal{E}_1^1(t)$, as $t \rightarrow \rho^{-}$. As $t \rightarrow \rho^{-}$, we have $1 - \mathcal{U}(t) \sim \gamma (1 - t/\rho)^{\frac{1}{2}}$  (eq.~(\ref{eq:utasymp})) and $\mathcal{U}(t) \rightarrow 1$, so that as $t \rightarrow \rho^{-}$, 
\begin{equation}
\label{eq:e11tul}
    \mathcal{E}^1_1(t) \sim \frac{\rho}{2\gamma^3(1 - \frac{t}{\rho})^{\frac{3}{2}}}.
\end{equation}

To derive the growth of the coefficients $E_{n,1}^1$, we use the transfer theorems summarized in eq.~(1) in Section VII.1 of \cite{FlajoletAndSedgewick09} (p.~441; see also Corollary VI.1, p.~392): if $f(t) = \sum_{n=0}^\infty f_nt^n$ for a function $f$ that is $\Delta$-analytic (a technical condition that holds in all our applications), $f(t) \sim (1-t/\nu)^{-\alpha}$ as $t \rightarrow \nu$ for $t \in \Delta$, and $\alpha$ is not a non-positive integer, then
\begin{equation}
f_n = [t^n] f(t) \sim \frac{1}{\Gamma(\alpha)}n^{\alpha-1}\nu^{-n}.
    \label{eq:transferth}
\end{equation}
Using eq.~(\ref{eq:transferth}) with eq.~(\ref{eq:e11tul}), as $n \rightarrow \infty$, we get
\begin{equation}
\label{eq:e1nonegall}
 E_{n,1}^1 = [t^n] \mathcal{E}^1_1(t) \sim \frac{\rho}{2\gamma^3}\frac{n^{\frac{1}{2}}\rho^{-n}}{\Gamma(\frac{3}{2})} = \frac{1} {\gamma^3\sqrt{\pi}}n^{\frac{1}{2}}\rho^{-n+1},
\end{equation}
where we note that $\Gamma(\frac{3}{2}) = \sqrt{\pi}/{2}$. Via eq.~(\ref{eq:En1}), the asymptotic expression for $E_{n,1}^1$ satisfies
$$E_{n,1}^1 \sim \rho E_{n,1} \sim \rho \tilde{E}_{n,1}.$$

The first terms of $E_{n,1}^1$ are calculated in Section \ref{sec:exactul1tcgt}. They appear in Table \ref{table:E1ng}.

\subsection{Two galls} 
\label{sec:1tcgt2g}

For two galls, we have three cases (Figure \ref{fig:1tcgttwogalls}). The first two are analogous to the first two cases for general galled trees (Figure \ref{fig:ggttwogalls}). If there is no root gall, then as in general galled trees, both galls are contained within the two subtrees of the root node. Either one of the subtrees contains both galls (component 1 in eq.~(\ref{eq:1tcgt2gsm}) and Figure \ref{fig:1tcgttwogalls}), or each subtree contains one gall (component 2). 

\begin{figure}[tb]
    \centering
    \includegraphics[width=6.3cm]{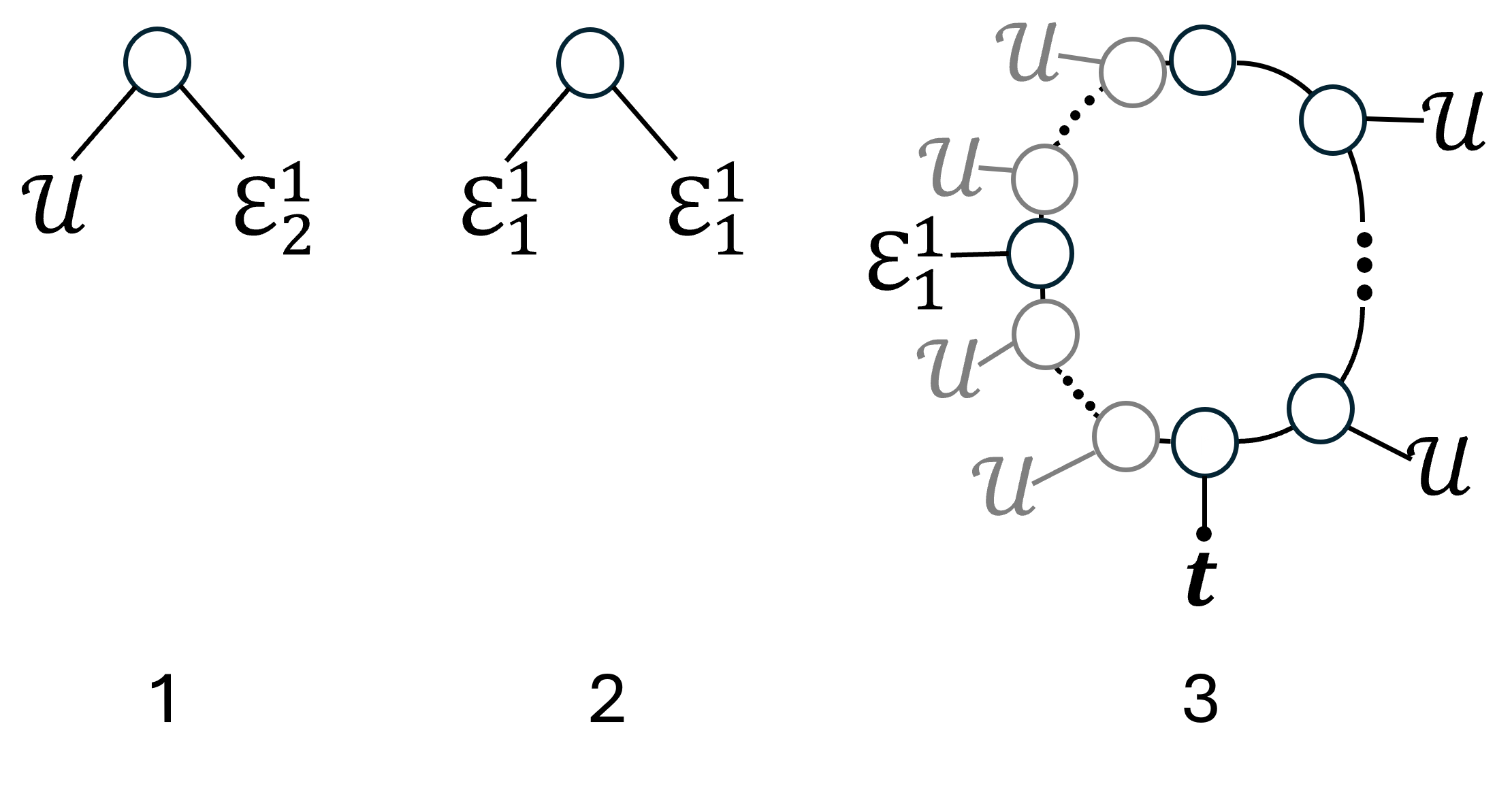}
    \vspace{-0.3cm}
    \caption{The three structures possible for a (non-plane) unlabeled simplex time-consistent galled tree with exactly two galls.}
    \label{fig:1tcgttwogalls}
\end{figure}

If there is a root gall (component 3 in eq.~(\ref{eq:1tcgt2gsm}) and Figure \ref{fig:1tcgttwogalls}), then by the time-consistency property, two non-empty sequences of nodes lead from the root node to the reticulation node (from which a leaf descends). A simplex time-consistent galled tree descends from one of the nodes in the sequence, and trees without galls descend from all the remaining nodes. Hence, one of the sequences is a nonempty sequence of trees with no galls and the other is a simplex time-consistent galled tree with one gall, flanked by two (potentially empty) sequences of trees with no galls. In symbolic notation,
\begin{equation}
    \mathcal{E}_2^1 = \{\circ\} \times \Big[\underbrace{ \mathcal{U}\times \mathcal{E}_2^1 }_{1} 
    \textrm{  }\dot{\cup}\textrm{  } 
    \underbrace{\text{MSET}_2(\mathcal{E}_1^1)}_{2} 
    \textrm{  }\dot{\cup}\textrm{  } 
    \underbrace{\{\square\}\times \big(\text{SEQ}(\mathcal{U}) \times \mathcal{E}_1^1 \times \text{SEQ}(\mathcal{U}) \big) \times \text{SEQ}^{+}(\mathcal{U})}_{3} \Big].
    \label{eq:1tcgt2gsm}
\end{equation}

Translating to a generating function $\mathcal{E}_2^1(t) = \sum_{n=0}^\infty E_{n,2}^1 t^n$, we have
$$\mathcal{E}_2^1(t) = \mathcal{U}(t) \, \mathcal{E}_2^1(t) + \frac{1}{2}[\mathcal{E}_1^1(t)^2 + \mathcal{E}_1^1(t^2)] + t \bigg(\frac{\mathcal{E}_1^1(t)}{[1 - \mathcal{U}(t)]^2}\bigg) \bigg(\frac{\mathcal{U}(t)}{1 - \mathcal{U}(t)}\bigg).$$ 
Isolating $\mathcal{E}_2^1{(t)}$ gives 
\begin{equation}
\label{eq:E21}
\mathcal{E}_2^1(t) = \frac{\mathcal{E}_1^1(t)^2}{2[1 - \mathcal{U}(t)]} + \frac{\mathcal{E}_1^1(t^2)}{2[1-\mathcal{U}(t)]} + \frac{t \, \mathcal{E}_1^1(t) \, \mathcal{U}(t)}{[1 - \mathcal{U}(t)]^4}.
\end{equation}

To derive asymptotics of $\mathcal{E}^1_2(t)$, note that as $t \rightarrow \rho^{-}$, $\mathcal{E}_1^{{1}} (t)$ has order $(1 - t/\rho)^{-\frac{3}{2}}$ (eq.~(\ref{eq:e11tul})) and $1 - \mathcal{U}(t)$ has order $(1 - t/\rho)^{\frac{1}{2}}$ (eq.~(\ref{eq:utasymp})). Hence, the first and third terms in eq.~(\ref{eq:E21}) have the same order---in particular, $(1 - t/\rho)^{-\frac{7}{2}}$---dominating the second term. From the eqs.~(\ref{eq:utasymp}) and (\ref{eq:e11tul}), as $t \rightarrow \rho^{-}$, we get
\begin{equation}
\label{eq:e12tul}
    \mathcal{E}_2^1(t) \sim \frac{\frac{\rho^2}{4\gamma^6(1 - \frac{t}{\rho})^3}}{2\gamma(1 - \frac{t}{\rho})^{\frac{1}{2}}} + \frac{\rho \frac{\rho}{2\gamma^3(1 - \frac{t}{\rho})^{\frac{3}{2}}}}{\gamma^4(1 - \frac{t}{\rho})^2} = 
    \frac{5\rho^2}{8\gamma^7(1 - \frac{t}{\rho})^{\frac{7}{2}}}.
\end{equation}

Using eq.~(\ref{eq:transferth}) with $\alpha = \frac{7}{2}$ this time, as $n\rightarrow \infty$,
\begin{equation}
\label{eq:e1ntwogalls}
E_{n,2}^1 = [t^n] \mathcal{E}_2^1(t) \sim \frac{5\rho^2}{8\gamma^7} \frac{n^{\frac{5}{2}}\rho^{-n}}{\Gamma(\frac{7}{2})} = \frac{1}{3\gamma^7\sqrt{\pi}}n^{\frac{5}{2}}\rho^{-n+2}    
\end{equation}
where we note that $\Gamma(\frac{7}{2}) = \frac{15}{8}\sqrt{\pi}$. Via eq.~(\ref{eq:En2}), the asymptotic expression for $E_{n,2}^1$ satisfies
$$E_{n,2}^1 \sim \rho^2 E_{n,2} \sim \rho^2 \tilde{E}_{n,2}.$$

The first terms of $E_{n,2}^1$ are calculated in Section \ref{sec:exactul1tcgt}, and they appear in Table \ref{table:E1ng}.

\subsection{Bivariate generating function} 
\label{sec:1tcgtbgf}

To derive the generating function $\mathcal{E}^1_g(t) = \sum_{n=0}^{\infty} E_{n,g}^1t^n$ for any fixed $g \geq 1$, we first turn our attention to the bivariate generating function $\mathcal{G}^1(t,u) = \sum_{n=0}^\infty \sum_{m=0}^\infty E^1_{n,m} t^n u^m$, where the coefficient of $t^n u^m$ denotes the number of unlabeled simplex time-consistent galled trees with $n$ leaves and $m$ galls. 

\begin{figure}[tb]
    \centering
    \includegraphics[width=6.3cm]{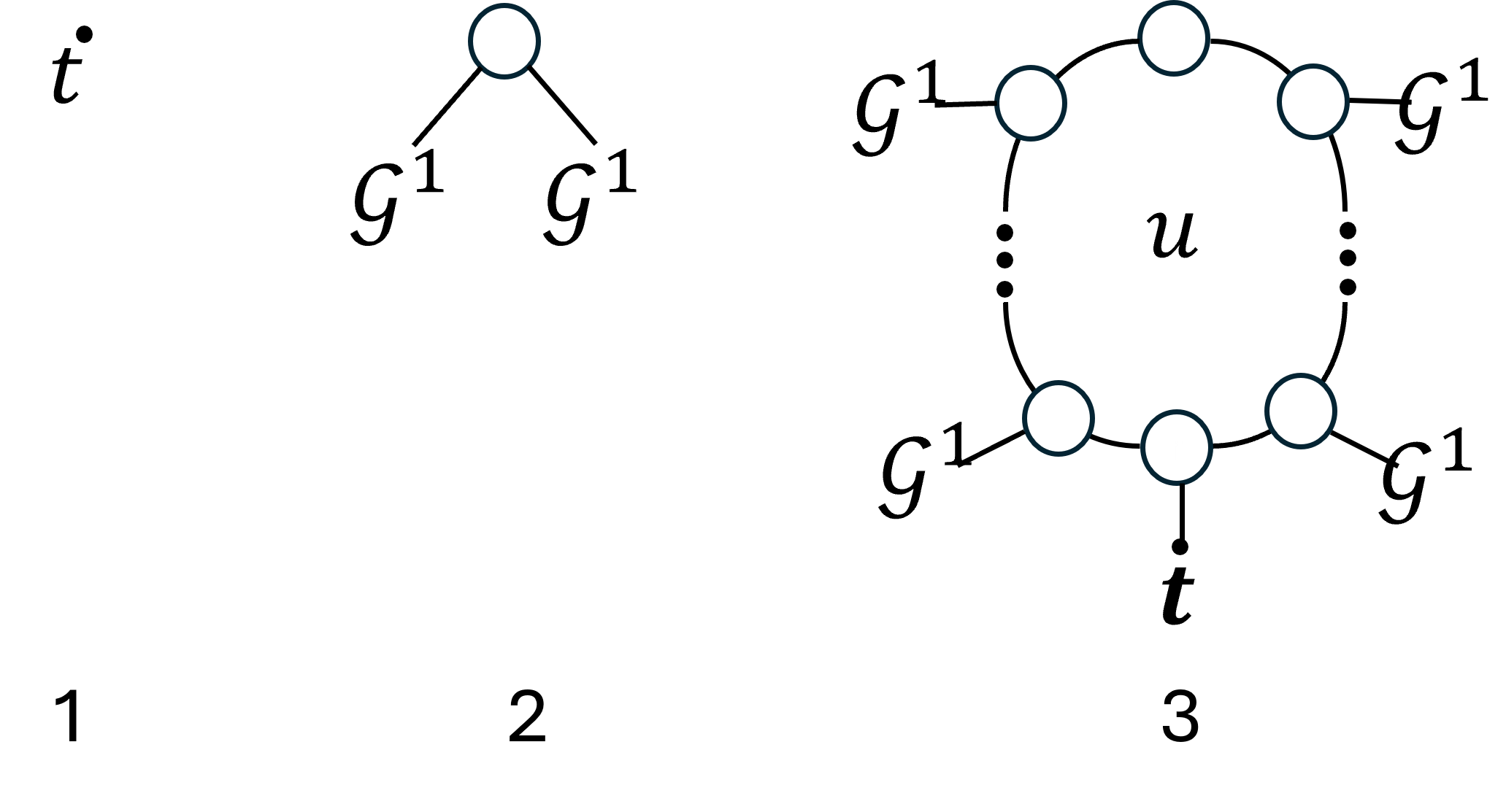}
    \vspace{-0.4cm}
    \caption{The three structures possible for a (non-plane) unlabeled simplex time-consistent galled tree.}
    \label{fig:1tcgtbivariate}
\end{figure}

A simplex time-consistent galled tree has three possibilities (Figure \ref{fig:1tcgtbivariate}). First, it can be one leaf (component 1 in eq.~(\ref{eq:g1tuul}) and Figure \ref{fig:1tcgtbivariate}) or not. Otherwise, if there is no root gall, then the root has two subtrees, each of which is a simplex time-consistent galled tree (component 2). If there is a root gall, then there are two nonempty sequences of nodes leading from the root to the reticulation node, from each of which a simplex time-consistent galled tree can descend (component 3). In the symbolic method construction, we take care to include the leaf descended from the reticulation node and the root gall itself: 
\begin{equation}
    \mathcal{G}^1 = \underbrace{\{\square\}}_{1} 
    \textrm{  }\dot{\cup}\textrm{  } 
    \{\circ\} \times  \Big[ {\underbrace{\text{MSET}_2(\mathcal{G}^1)}_{2}} 
    \textrm{  }\dot{\cup}\textrm{  }
    \underbrace{\mu \times \{\square\} \times \text{MSET}_2 \big(\text{SEQ}^+(\mathcal{G}^1)\big)}_{3} \Big].
    \label{eq:g1tuul}
\end{equation}
Converting to a generating function, we get 
\begin{equation}
\mathcal{G}^1(t, u) = t + \frac{1}{2}[\mathcal{G}^1(t, u)^2 + \mathcal{G}^1(t^2, u^2)] + \frac{ut}{2}\bigg[\bigg(\frac{\mathcal{G}^1(t, u)}{1 - \mathcal{G}^1(t,u)}\bigg)^2 + \frac{\mathcal{G}^1(t^2, u^2)}{1 - \mathcal{G}^1(t^2, u^2)}\bigg].
\label{eq:g1tubivgf}
\end{equation}

\subsection{Fixed number of galls} 
\label{sec:fixedul1tcgt}

From the bivariate generating function, we derive the generating function $\mathcal{E}_g^1(t)$ counting unlabeled simplex time-consistent galled trees with $g$ galls. As in the case of general galled trees, for each $g \geq 1$, 
\begin{equation}
\mathcal{E}^1_g(t) = \frac{1}{g!} \bigg(\frac{\partial^g}{\partial u^g} \mathcal{G}^1\bigg)(t, 0).
\label{eq:gderiv}
\end{equation}
To calculate this derivative, we make use of Leibniz's product rule and Fa\`a di Bruno's formula, as in the case of general galled trees (Section \ref{sec:fixedulggt}).

We sum the appropriate derivatives of the terms in eq.~(\ref{eq:g1tubivgf}). Derivatives with respect to $u$ of the initial term, $t$, vanish. For the next two terms, from eqs.~(17) and (18) of \cite{AgranatTamirEtAl25}, 
\begin{align}
\label{eq:term1}
\frac{1}{2g!} D_u^g[\mathcal{G}(t, u)^2] \big|_{u=0} &=\frac{1}{2}\sum_{\ell=0}^g \mathcal{E}^1_{\ell}(t) \, \mathcal{E}^1_{g-\ell}(t) \\
\label{eq:term2}
\frac{1}{2g!} D_u^g[\mathcal{G}_1(t^2, u^2)] \big|_{u=0} &= \frac{1}{2}\mathcal{E}^1_{\frac{g}{2}}(t^2),
\end{align}
understanding $\mathcal{E}^1_{\frac{g}{2}}(t^2) = 0$ if $g$ is odd.

The derivative of the next term, $({ut}/{2}) {\mathcal{G}^1(t,u)^2}/{[1 - \mathcal{G}^1(t,u)]^2}$, is more involved. First, in Leibniz's product rule (eq.~(\ref{eq:leibniz})), at $u=0$, the only term whose $g$th derivative does not vanish places one derivative on $\frac{ut}{2}$ and the other $g-1$ derivatives on ${\mathcal{G}^1(t,u)^2}/{[1 - \mathcal{G}^1(t, u)]^2}$. Hence, we have 
$$\frac{1}{g!} D_u^g \bigg( \frac{ut}{2} \frac{\mathcal{G}^1(t,u)^2}{[1 - \mathcal{G}^1(t,u)]^2} \bigg) \bigg\rvert_{u=0} = \frac{t}{2(g-1)!} D_u^{g-1} \bigg(\frac{\mathcal{G}^1(t,u)^2}{[1 - \mathcal{G}^1(t,u)]^2}\bigg) \bigg\rvert_{u=0}.$$ 

Using Fa\`a di Bruno's formula (eq.~(\ref{eq:fdb})) with $f(x) = {x^2}/{(1-x)^2}$ and $h(u) = \mathcal{G}^1(t,u)$ (with $u=0$ and $h(0) = \mathcal{U}(t)$) and noting that for any $\ell \geq 1$, $\frac{1}{\ell!} {D_u^{\ell} h(u)} |_{u=0} = \mathcal{E}^1_{\ell}(t)$ (eq.~(\ref{eq:gderiv})), we can write 
$$\frac{1}{g!} D_u^g \left(\frac{ut}{2} \frac{\mathcal{G}^1(t,u)^2}{[1 - \mathcal{G}^1(t,u)]^2}\right){\bigg\vert_{u=0}} = \frac{t}{2} \sum_{k_1 + 2k_2 + \ldots \atop + (g-1)k_{g-1} = g-1} \frac{1}{k_1! \, k_2! \cdots k_{g-1}!} D_x^{k_1 + k_2 + \ldots + k_{g-1}} \left(\frac{x^2}{(1-x)^2}\right) \bigg\rvert_{x = \mathcal{U}(t)} \prod_{\ell=1}^{g-1} \mathcal{E}^1_{\ell}(t)^{k_{\ell}}.$$

It is straightforward to show by induction that for any $k \geq 1$, 
$$\frac{1}{k!} D_x^k \left[\frac{x^2}{(1-x)^2}\right] = \frac{2x + (k-1)}{(1-x)^{k+2}}.$$ 
We then have 
\begin{equation}\frac{1}{g!} D_u^g \bigg(\frac{ut}{2} \frac{\mathcal{G}^1(t,u)^2}{[1 - \mathcal{G}^1(t,u)]^2} \bigg) {\bigg\vert_{u=0}}  = \frac{t}{2} \sum_{k_1 + 2k_2 + \ldots \atop + (g-1)k_{g-1} = g-1} \frac{\left(\sum_{\ell=1}^{g-1} k_{\ell}\right)!}{k_1! \, k_2! \cdots k_{g-1}!} \frac{2\mathcal{U}(t) + (\sum_{\ell=1}^{g-1} k_{\ell}) - 1}{[1 - \mathcal{U}(t)]^{(\sum_{\ell=1}^{g-1} k_{\ell}) + 2}} \prod_{\ell=1}^{g-1} \mathcal{E}^1_{\ell}(t)^{k_{\ell}}.
\label{eq:term3}
\end{equation}

For the last term of $G^1(t,u)$, we follow a similar approach. In  Leibniz's product rule, at $u=0$, all terms disappear except the term with one derivative of $\frac{ut}{2}$ and the other $g-1$ of $\mathcal{G}^1(t^2, u^2)/[1 - \mathcal{G}^1(t^2, u^2)]$. Hence,  
$$\frac{1}{g!} D_u^g \left(\frac{ut}{2} \frac{\mathcal{G}^1(t^2, u^2)}{1 - \mathcal{G}^1(t^2. u^2)}\right) \bigg\rvert_{u=0} = \frac{t}{2(g-1)!} D_u^{g-1} \left(\frac{\mathcal{G}^1(t^2, u^2)}{1 - \mathcal{G}^1(t^2. u^2)}\right) \bigg\rvert_{u=0}.$$
By eq.~(22) of \cite{AgranatTamirEtAl25}, $D_u^{g-1} \big({\mathcal{G}^1(t^2, u^2)} /[1 - \mathcal{G}^1(t^2, u^2)] \big) \big|_{u=0} \neq 0$ only if $g-1$ is even, and  
$$D_u^{g-1} \left(\frac{\mathcal{G}^1(t^2, u^2)}{1 - \mathcal{G}^1(t^2, u^2)}\right)\bigg\rvert_{u=0} = \sum_{2r_2 + 4r_4 + \ldots \atop + (g-1)r_{g-1} = g-1} \frac{(g-1)!}{r_2! \, r_4! \cdots r_{g-1}!} D_x^{r_1 + r_2 + \ldots + r_{g-1}} \left(\frac{x}{1-x}\right) \bigg\rvert_{x = \mathcal{U}(t^2)} \prod_{\ell=1}^{g-1} \frac{1}{\ell!} D_u^{\ell} \mathcal{G}^1(t^2, u^2) \big|_{u=0}.$$

Now, note that, as we can prove by induction, for each $k \geq 1$, $$D_x^k \left(\frac{x}{1-x}\right) = \frac{k!}{(1-x)^{k+1}}.$$ 
In addition, for even $m$, 
$$D_u^m [\mathcal{G}_1(t^2, u^2) \big] |_{u=0} = m! \, \mathcal{E}^1_{\frac{m}{2}}(t^2).$$
Inserting these, we get, for odd $g$ (and even $g-1$), 
\begin{equation}
D_u^{g-1} \bigg(\frac{\mathcal{G}^1(t^2, u^2)}{1 - \mathcal{G}^1(t^2, u^2)}\bigg) \bigg\rvert_{u=0} = \sum_{2r_2 + 4r_4 + \ldots \atop + (g-1)r_{g-1} = g-1} \frac{(g-1)!}{r_2! \, r_4! \cdots r_{g-1}!} \Bigg(\frac{\Big(\sum_{\ell=1}^{\frac{g-1}{2}} r_{2\ell}\Big)!}{[1 - \mathcal{U}(t^2)]^{(\sum_{\ell=1}^{(g-1)/2} r_{2\ell}) + 1}}\Bigg) \prod_{\ell=1}^{\frac{g-1}{2}} \mathcal{E}_1^{\ell}(t^2)^{k_{\ell}}.
\label{eq:term4}
\end{equation}

Summing the derivatives of individual terms in eqs.~(\ref{eq:term1}), (\ref{eq:term2}), (\ref{eq:term3}), and (\ref{eq:term4}), we get
\begin{align}
\mathcal{E}_g^1(t) & = \frac{1}{2[1-\mathcal{U}(t)]}\Bigg[ \bigg( \sum_{\ell=1}^{g-1} \mathcal{E}^1_{\ell}(t) \, \mathcal{E}^1_{g-\ell}(t) \bigg) + \delta_{0,g \textrm{ mod } 2}\cdot\mathcal{E}^1_{\frac{g}{2}}(t^2) \nonumber \\ 
& \quad + t \sum_{k_1 + 2k_2 + \ldots \atop + (g-1)k_{g-1} = g-1} 
\binom{k_1+k_2+\ldots+k_{g-1}}{k_1,k_2, \ldots ,k_{g-1}}
\frac{2\mathcal{U}(t) + (\sum_{\ell=1}^{g-1} k_{\ell}) - 1}{[1 - \mathcal{U}(t)]^{(\sum_{\ell=1}^{g-1} k_{\ell}) + 2}} \prod_{\ell=1}^{g-1} \mathcal{E}^1_{\ell}(t)^{k_{\ell}} \nonumber \\ 
& \quad + \delta_{1,g \textrm{ mod } 2}\Bigg(t \sum_{2r_2 + 4r_4 + \ldots \atop + (g-1)r_{g-1} = g-1} 
\binom{r_2+r_4+\ldots+r_{g-1}}{r_2,r_4,\ldots,r_{g-1}}
\bigg(\frac{1}{[1 - \mathcal{U}(t^2)]^{(\sum_{\ell=1}^{(g-1)/2} r_{2\ell}) + 1}}\bigg) \prod_{\ell=1}^{\frac{g-1}{2}} \mathcal{E}^1_{\ell}(t^2)^{r_{\ell}}\Bigg)\Bigg].
\label{eq:gf3}
\end{align}

Using the generating function in eq.~(\ref{eq:gf3}), we can derive the following result. The asymptotic analysis of $\mathcal{E}_g^1(t)$ largely follows that of $\mathcal{E}_g(t)$ in \cite{AgranatTamirEtAl24b}, and by the asymptotic equivalence of $\mathcal{E}_g(t)$ and $\tilde{\mathcal{E}}_g(t)$ in Section \ref{sec:fixedulggt}, that of $\tilde{\mathcal{E}}_g(t)$. However, an extra factor of $\rho^g$ appears in the analysis of $\mathcal{E}_g^1(t)$. 
\begin{theorem}
\label{thm:1}
As $t\rightarrow \rho^-$, the generating function $\mathcal{E}_g^1(t)$ for the number of unlabeled simplex time-consistent galled trees with $g$ galls, $g \geq 1$, satisfies \begin{equation}
    \mathcal{E}^1_g(t) \sim \beta_g \frac{\rho^g}{\gamma^{4g-1}(1 - \frac{t}{\rho})^{2g - \frac{1}{2}}}
    \label{eq:Eg1asymp}
\end{equation}
for a constant $\beta_g$. Furthermore, for $g \geq 2$, the constants $\beta_g$ obey the recurrence relation 
\begin{equation}
\label{eq:Bgrecurr}
\beta_g = \bigg( \frac{1}{2} \sum_{\ell=1}^{g-1} \beta_{\ell}\beta_{g-\ell} \bigg) 
+ \frac{1}{2} \sum_{k_1 + 2k_2 + \ldots \atop + (g-1)k_{g-1} = g-1} 
\binom{k_1+ k_2 +\ldots+ k_{g-1}}{k_1, k_2 ,\ldots, k_{g-1}}
{\bigg[\bigg(\sum_{\ell=1}^{g-1}k_{\ell} \bigg) + 1\bigg]} \prod_{\ell=1}^{g-1} \beta_{\ell}^{k_{\ell}},
\end{equation}
with $\beta_1 = \frac{1}{2}$.
\end{theorem}

\begin{proof} 
We perform strong induction on $g$. The base case of $g=1$  calculated in eq.~(\ref{eq:e11tul}) follows eq.~(\ref{eq:Eg1asymp}), with $\beta_1=\frac{1}{2}$.
The $g=2$ case in eq.~(\ref{eq:e12tul}) follows eq.~(\ref{eq:Eg1asymp}), with $\beta_2 = \frac{5}{8} = \frac{1}{2} ( \frac{1}{2} \frac{1}{2} ) + \frac{1}{2} (1)(2) \frac{1}{2}$ as in eq.~(\ref{eq:Bgrecurr}).

Now, assume that, for some $g \geq 2$, the claim holds for all $k$, $1 \leq k \leq g-1$. We show that the claim holds for $k=g$ as well. We first determine which terms in eq.~(\ref{eq:gf3}) contribute to the asymptotic approximation, considering the exponent of $1 - \frac{t}{\rho}$ in the denominator.

Assuming the inductive hypothesis, the second and fourth terms, those involving functions of $t^2$, do not contribute, as $\rho^2 < \rho$ and $\rho^2$ is not a singularity of $\mathcal{E}_g^1(t)$. We examine the first and third terms.

For each $\ell$, $1 \leq g-1$, by the inductive hypothesis and eq.~(\ref{eq:utasymp}), the exponent of $1-\frac{t}{\rho}$ in the denominator is ${\mathcal{E}^1_{\ell}(t) \, \mathcal{E}^1_{g-\ell}(t)} / {[2(1 - \mathcal{U}(t))]}$ contributes 
$$\bigg(2\ell - \frac{1}{2}\bigg) + \bigg[2(g-\ell) - \frac{1}{2}\bigg] + \frac{1}{2} = 2g - \frac{1}{2}.$$ 
Similarly, by the inductive hypothesis and eq.~(\ref{eq:utasymp}), each term in the third sum contributes 
$$\frac{1}{2} \bigg[\bigg(\sum_{\ell=1}^{g-1} k_{\ell} \bigg) + {3}\bigg] + \sum_{\ell=1}^{g-1} \bigg(2\ell - \frac{1}{2}\bigg) k_{\ell} = 2\bigg(\sum_{\ell=1}^{g-1} \ell k_{\ell} \bigg) + \frac{3}{2} = 2(g-1) + \frac{3}{2} = 2g - \frac{1}{2}.$$

For the exponent of $\gamma$ in the asymptotic expansion, we note that each factor of $(1-\frac{t}{\rho})^{\frac{1}{2}}$ is accompanied by a factor of $\gamma$, so that the exponent of $\gamma$ in the asymptotic expansion is $2(2g - \frac{1}{2}) = 4g - 1$. 

For the exponent of $\rho$, by the inductive hypothesis, for the first term in the asymptotic approximation, $\mathcal{E}^1_{\ell}(t) \, \mathcal{E}^1_{g-\ell}(t)$ has $\ell+(g-\ell)=g$ as the exponent of $\rho$. In third term, the exponent of $\rho$ in $\prod_{\ell=1}^{g-1} \mathcal{E}^1_{\ell}(t)^{k_{\ell}}$ is $\sum_{\ell=1}^{g-1} \ell k_{\ell} = g-1$, which, when combined with the $t$ term from outside the sum, as $t \rightarrow \rho^-$, leads to an exponent equal to $g$. We have therefore demonstrated that for $g \geq 1$ galls, each term that contributes to the asymptotic approximation is proportional to $\rho^g/\big[\gamma^{4g-1}(1 - \frac{t}{\rho})^{2g - \frac{1}{2}} \big]$.

To compute the proportionality constant, $\beta_g$, we note that $\lim_{t \rightarrow \rho^-} \mathcal{U}(t)=1$, and consequently, the first and third terms of eq.~(\ref{eq:gf3}) produce eq.~(\ref{eq:Bgrecurr}).

\end{proof}

\bp
\label{prop:2}
The constants $\{\beta_g\}_{g \geq 1}$ satisfy $2^{2g-1}\beta_g=C_{2g-1}$, where $C_{m} = \frac{1}{m+1} {2m \choose m}$ is a Catalan number.
\ep
\bpf
We show an equality between $\mathfrak{D}(t) = \sum_{g=1}^\infty 2^{2g-1}\beta_{g}t^{2g-1}$ and the odd part of the generating function of the Catalan numbers $\mathcal{C}_O(t) = \sum_{g=1}^\infty C_{2g-1}t^{2g-1}$.

First, following claim 2 in the proof of Proposition 1 of \cite{AgranatTamirEtAl25}, we rewrite $\sum_{\ell=1}^{g-1}k_{\ell}=\ell$, and we can then replace the sum $\sum_{k_1+2k_2+\ldots +(g-1)k_{g-1}=g-1} \binom{k_1+k_2+\ldots +k_{g-1}}{k_1,k_2,\ldots,k_{g-1}}$ with a sum over compositions $C(g-1, \ell)$ of $g-1$ into $\ell$ positive integer parts, $\sum_{\ell=1}^{g-1}\sum_{d\in C(g-1,\ell)}$. Eq.~(\ref{eq:Bgrecurr}) becomes 
\begin{equation*}
    \beta_g = \bigg(  \frac{1}{2} \sum_{\ell=1}^{g-1}\beta_{\ell}\beta_{g-\ell} \bigg) + \frac{1}{2} \bigg[ (\ell+1)\sum_{\mathbf{d}\in C(g-1,\ell)}\prod_{j=1}^{\ell}\beta_{d_j}\bigg].
\end{equation*}

Next,
\begin{align*}
    \mathfrak{D}(t) &= 2^{2-1}\frac{1}{2}t
    + \sum_{g=2}^\infty \bigg(\sum_{\ell=1}^{g-1}2^{2g-2}\beta_{\ell}\beta_{g-\ell}\bigg) t^{2g-1} + \sum_{g=2}^\infty \bigg[\sum_{\ell=1}^{g-1}(\ell+1)2^{2g-2}\sum_{\mathbf{d}\in C(g-1,\ell)}\prod_{j=1}^{\ell}\beta_{d_j}\bigg]t^{2g-1} \\ \nonumber
  &= t+\bigg[ \sum_{\ell=1}^\infty 2^{2\ell-1}\beta_{\ell}t^{2\ell-1}\sum_{g=\ell+1}^\infty 2^{2(g-\ell)-1}\beta_{g-\ell}t^{2(g-\ell)-1} \bigg]t \\ \nonumber
  &  \quad + \bigg[ \sum_{\ell=1}^\infty (\ell+1)(2t)^{\ell} 
  \sum_{g=\ell+1}^\infty \sum_{\mathbf{d}\in C(g-1,\ell)} \prod_{j=1}^{\ell} 2^{2d_j-1}\beta_{d_j}t^{2d_j-1} \bigg]t  \\ \nonumber
  &= t+t\mathfrak{D}(t)^2+t\sum_{\ell=1}^\infty (\ell+1)[2t\mathfrak{D}(t)]^{\ell} \\ \nonumber
  &= t+t\mathfrak{D}(t)^2+\frac{2t^2\mathfrak{D}(t)}{[1-2t\mathfrak{D}(t)]^2}+\frac{2t^2\mathfrak{D}(t)}{1-2t\mathfrak{D}(t)}.
\end{align*}
We have shown in Proposition 9 of \cite{AgranatTamirEtAl24b} that this implicit expression gives the generating function for the Catalan numbers with odd indices. 
\epf

\begin{corollary} 
 \label{cor:eg1compapprox}
The generating function $\mathcal{E}_g^1(t)$ for the number of unlabeled simplex time-consistent galled trees with $g$ galls satisfies as $t \rightarrow \rho^-$
    \begin{equation}
    \mathcal{E}_g^1(t) \sim \frac{(4g-3)!! \, \rho^g}{(2g)! \, \gamma^{4g-1} (1-\frac{t}{\rho})^{2g-\frac{1}{2}}}.
    \end{equation}
\end{corollary}
\begin{proof}
    In Theorem \ref{thm:1}, we insert the constant $\beta_g$ from Proposition \ref{prop:2}. The result follows from the identity $C_{2g-1}=[2^{2g-1}(4g-3)!!]/(2g)!$ for the Catalan numbers.
\end{proof}

\medskip
Using Corollary \ref{cor:eg1compapprox} in conjunction with eq.~(\ref{eq:transferth}), we obtain the following proposition:

\begin{proposition} 
The asymptotic growth of the number $E_{n,g}^1$ of unlabeled simplex time-consistent galled trees with $n$ leaves and a fixed number of galls $g \geq 1$ satisfies 
\begin{equation}
    E_{n,g}^1 \sim \frac{2^{2g-1}}{(2g)! \, \gamma^{4g-1}\sqrt{\pi}}n^{2g-\frac{3}{2}}\rho^{-n+g}.
    \label{eq:Eng1}
\end{equation}
\end{proposition}

\begin{proof}
We apply eq.~(\ref{eq:transferth}) to the expression in Corollary \ref{cor:eg1compapprox} and simplify using the identity $\Gamma\left(2g - \frac{1}{2}\right) =  \sqrt{\pi} [(4g-3)!!]/2^{2g-1}$.
\end{proof}

\medskip
Via eq.~(\ref{eq:Eng}), the asymptotic expression for $E_{n,g}^1$ satisfies $$E_{n,g}^1 \sim \rho^g E_{n,g} \sim \rho^g \tilde{E}_{n,g}.$$
Values of $E_{n,g}^1$ with small $n$ and $g$ are calculated in Section \ref{sec:exactul1tcgt}. They appear in Table \ref{table:E1ng}.

\subsection{Arbitrary number of galls} 
\label{sec:allul1tcgt}

Next, we study the generating function of unlabeled simplex time-consistent galled trees with an arbitrary number of galls, $\mathcal{A}^1(t)=\sum_{n=0}^\infty A^1_n t^n$. A tree has three possibilities: (1) it is a single node (component 1 in eq.~(\ref{eq:a1t}) and Figure \ref{fig:1tcgtbivariate}); (2) it has no root gall, and the two subtrees of the root are both simplex time-consistent galled trees (component 2); (3) it has a root gall, from which two nonempty sequences of nodes lead from the root to the reticulation node. From each node in the sequence, a simplex time-consistent galled tree descends (component 3). In symbolic notation,
\begin{equation}
\mathcal{A}^1 = \underbrace{\{\square\}}_{1} 
\textrm{  }\dot{\cup}\textrm{  }
\{\circ\} \times \big[\underbrace{\text{MSET}_2(\mathcal{A}^1)}_{2} 
\textrm{  }\dot{\cup}\textrm{  } 
\underbrace{\{\square\} \times \text{MSET}_2 \big(\text{SEQ}^{+}(\mathcal{A}^1)\big)}_{3}  \big].
    \label{eq:a1t}
\end{equation}
Translating into a generating function, we get 
\begin{equation}
    \mathcal{A}^1(t) = t + \frac{1}{2}[\mathcal{A}^1(t^2) + \mathcal{A}^1(t)^2] + \frac{t}{2} \bigg[ \bigg( \frac{\mathcal{A}^1(t)}{1 - \mathcal{A}^1(t)} \bigg)^2 + \frac{\mathcal{A}^1(t^2)}{1 - \mathcal{A}^1(t^2)}\bigg].
    \label{eq:A1tgf}
\end{equation}

To derive asymptotics for $\mathcal{A}^1(t)$ we use the \emph{asymptotics of implicit tree-like classes} theorem~\citep{MeirAndMoon89b, MeirAndMoon89a}. According to the theorem~\citep[pp.~467-468]{FlajoletAndSedgewick09}, if $\mathcal{A}^1(t)$ belongs to the \emph{smooth implicit-function schema}, then as $n \rightarrow \infty$,  we conclude $A^1_n \sim [\delta/(2\sqrt{\pi})] n^{-\frac{3}{2}}r^{-n}$. In this expression, $r$ is the radius of convergence of $\mathcal{A}^1(t)$ and $\delta = \sqrt{2r\phi_t(r,s)/\phi_{ww}(r,s)}$, where $\phi(t,w)=\sum_{n=0}^\infty \sum_{k=0}^\infty S_{n,k}t^nw^k$, $\mathcal{A}^1(t)=\phi \big(t,\mathcal{A}^1(t) \big)$, and $(r,s)$ solves the \emph{characteristic system}
\begin{align}
    s &= \phi(r,s)   \label{eq:charsys1} \\ 
    1 &= \phi_w(r,s).     \label{eq:charsys2}
\end{align}

\begin{proposition}
    $\mathcal{A}^1(t)$ belongs to the smooth implicit-function schema.
\end{proposition}

\begin{proof}
    Showing that $\mathcal{A}^1(t)$ belongs to the smooth implicit-function schema amounts to validating several conditions. (1) $\mathcal{A}^1(t)$ is analytic at $0$. (2) $A^1_n \geq 0$ for all $n \geq 0$. (3) $A^1_0=0$.
\begin{enumerate}[ {(}1{)} ]    \setcounter{enumi}{3}
\item There exists a function $\phi(t,w)=\sum_{n=0}^\infty \sum_{k=0}^\infty S_{n,k} t^n w^k$ such that $\mathcal{A}^1(t)=\phi \big(t,\mathcal{A}^1(t) \big)$ and $\phi(t,w)$ satisfies the following conditions:
    \begin{enumerate}
    \item $\phi(t, w)$ is analytic around $(0, 0)$.
    \item The $S_{n,k}$ are nonnegative for all $n,k \geq 0$.
    \item  $S_{0,1}\neq1$. 
    \item There exists $S_{n,k} > 0$ for some $k \geq 0$ and $n \geq 2$. 
    \item $S_{0,0}=0$.
    \item There exists a solution $(r_0,s_0)$ to the characteristic system (eqs.~(\ref{eq:charsys1}) and (\ref{eq:charsys2})) in the region around $(0,0)$ where $\phi$ is analytic. 
    \end{enumerate}
\end{enumerate}
 
For the first three conditions, (1) $\mathcal{A}^1(t)$ is analytic at $0$ from the form of the expression in eq.~(\ref{eq:A1tgf}). (2) As $A_n^1$ represents a count of a class of objects, $A_n^1 \geq 0$ for all $n \geq 0$. (3) $A^1_0 = 0$ because there are no simplex time-consistent galled trees with no leaves.
\begin{enumerate}[ {(}1{)} ]
    \setcounter{enumi}{3}
    \item We define $\phi(t,w)=\sum_{n=0}^\infty \sum_{k=0}^\infty S_{n,k} t^n w^k$ as 
    \begin{equation*}
 \phi(t,w) = \sum_{n=0}^\infty \sum_{k=0}^\infty S_{n,k}t^nw^k =   t + \frac{1}{2}w^2 + \frac{1}{2}\mathcal{A}^1(t^2) + \frac{t}{2} \bigg[ \bigg(\frac{w}{1-w} \bigg)^2 + \frac{\mathcal{A}^1(t^2)}{1 - \mathcal{A}^1(t^2)} \bigg],
\end{equation*} 
    so that $\mathcal{A}^1(t)=\phi \big(t,\mathcal{A}^1(t) \big)$. In addition,
    \begin{enumerate}
        \item $\phi(t, w)$ is analytic around $(0, 0)$ because it is a rational function in $w$ and $\mathcal{A}^1(t)$ has some positive radius of convergence: its coefficients are bounded above by those of $\mathcal{A}(t)$, the generating function for time-consistent galled trees. We have shown that $\mathcal{A}(t)$ has a positive radius of convergence in Section 5.2 of \cite{AgranatTamirEtAl24a}, and therefore $\mathcal{A}^1(t)$ does as well. 
    \item The $S_{n,k}$ are all nonnegative because $\mathcal{A}^1(t^2) \geq 0$ for all $t$ (and $\mathcal{A}^1(r^2)<1$, as $r^2$ is less than the radius of convergence $r$, $\mathcal{A}^1(t)$ is monotonically increasing for $t<r$, and $\mathcal{A}^1(t)<1$ for $t < r$).
    \item $S_{0,1}\neq1$ because $S_{0,1}$ is the coefficient of $t^0 w^1$, and it is equal to 0. 
    \item $S_{0,2}$ is the coefficient of $t^0 w^2$ and is equal to $\frac{1}{2}$, so that $S_{0,2} > 0$.
    \item $S_{0,0}=0$ because $A^1_0=0$ so that $\mathcal{A}^1(0)=0$, leading to $\phi(0,0)=0$.
    \item There is a solution to the characteristic system as we show below. 
    \end{enumerate}
\end{enumerate}

We denote $\mathcal{A}^1(r^2)=b$, and the characteristic system is: 
\begin{align}
    s &=   r + \frac{1}{2}s^2 + \frac{1}{2}b +\frac{1}{2} r\bigg[\bigg(\frac{s}{1-s} \bigg)^2 + \frac{b}{1 - b} \bigg] \label{eq:a1ulcs1}\\
    1 &= s + \frac{rs}{(1 - s)^3}. \label{eq:a1ulcs2}
\end{align}
The radius of convergence $r$ is positive; it is less than 1 because the radius of convergence of the generating function of unlabeled simplex time-consistent galled trees is bounded above by the radius of convergence for the generating function of unlabeled binary trees, or $\rho \approx 0.40270$ (Section \ref{sec:ggt0galls}). We therefore have $\mathcal{A}^1(r^2)\leq\mathcal{A}^1(r)$. Hence, $\mathcal{A}^1(r^2)$ converges quickly and can be estimated by $\sum_{n=0}^N A^1_n r^{2n}$ for some small $N$; we use $N=25$. We calculate $\{A^1_n\}_{n=0}^N$ using recursion with $A_0^1=0$ and $A_1^1=1$ (the trivial tree). For $n\geq2$, 
\begin{equation} \label{eq:exactulAtildn}
    A^1_n =\frac{1}{2}\bigg[ \bigg( \sum_{m=1}^{n-1}A^1_{m}A^1_{n-m} \bigg) +\delta_{0, n \textrm{ mod } 2}A^1_{\frac{n}{2}}
     + \bigg( \sum_{k=3}^n(k-2)\sum_{\mathbf{c}\in C(n-1,k-1)}\prod_{i=1}^{k-1}A^1_{c_i} \bigg)
     + \delta_{1, n \textrm{ mod } 2} \sum_{a=1}^{\frac{n-1}{2}}\sum_{\mathbf{c}\in C(\frac{n-1}{2},a)}\prod_{i=1}^aA^1_{c_i}\bigg].
\end{equation}
The first two terms of eq.~(\ref{eq:exactulAtildn}), in which the tree does not have a root gall, are equal to those in the recursion for time-consistent galled trees \citep[eqs.~(15) and (16)]{AgranatTamirEtAl24a}. If a tree has a root gall, then the relevant part in the recursion for time-consistent galled trees is $\frac{1}{2}[\sum_{k=3}^n (k-2) \sum_{\mathbf{c}\in C(n,k)}\prod_{i=1}^k A_{c_i}] 
+ \frac{1}{2}[\sum_{a=1}^{\lfloor \frac{n-1}{2} \rfloor}\sum_{\mathbf{c}\in C_p(n,2a+1)} \prod_{i=1}^{a+1} A^1_{c_i}]$. In this expression, $k$ denotes the number of non-root nodes in the root gall, $k-2$ of which can be the reticulation node, keeping in mind the time-consistency condition. The second term is concerned with symmetric trees: $a$ is the number of non-root nodes in the root gall on each side of the reticulation node, and $C_p(a,b)$ is the set of palindromic compositions of $a$ into $b$ positive integers. 
    
To complete the derivation of the terms associated with root galls in eq.~(\ref{eq:exactulAtildn}), we note that in simplex time-consistent galled trees, the child of the reticulation node in the root gall (and any other gall) is a leaf. We must therefore distribute the remaining $n-1$ leaves among the children of the other $k-1$ non-root nodes in the root gall. In addition, because exactly 1 leaf descends from the reticulation node, symmetry between the two sides of the root gall can occur only if $n$ is odd. In symmetric trees, $(n-1)/{2}$ leaves descend from each side of the reticulation node, and the tree is determined by the children of the nodes in the root gall on one side (say, the left side) of the reticulation node. The values of ${A}^1_n$ appear in Table {\ref{table:E1ng}}.

Returning to the characteristic system, we now have three equations with three unknowns $(r,s,b)$: eq.~(\ref{eq:a1ulcs1}), eq.~(\ref{eq:a1ulcs2}), and $b=\sum_{n=1}^{25}A^1_n(r^2)^n$. From eq.~(\ref{eq:a1ulcs2}), $r = {(1-s)^4}/{s}$. Substituting into eq.~(\ref{eq:a1ulcs1}), we get an equation with only $s$ and $b$ as unknowns. Inserting  values $b < 1$ in that equation (because $A^1_n(r^2) < 1$), we numerically solve for $s$. Plugging these solutions into $r = {(1-s)^4}/{s}$, we obtain possible values for $r$, which we then insert into $\sum_{n=1}^{25}A^1_n(r^2)^n$. These must coincide with the values of $b$ we used to calculate $s$. This process gives us $r\approx 0.2344$, $s \approx 0.4349$, and $b \approx 0.0584$. The identification of $(r,s)$ completes the proof that $\phi$ belongs to the smooth implicit-function schema.
\end{proof}

\medskip 
From the \emph{asymptotics of implicit tree-like classes} theorem, to find the approximation of $A^1_n$ as $n \rightarrow \infty$, we must calculate $\phi_{ww}(r,s)$ and $\phi_t(r,s)$: 
\begin{align*}
\phi_{ww}(t,w) &= 1+ \frac{t(1+2w)}{(1-w)^4} \nonumber \\
\phi_t(t,w) &= 1+tA^{1'}(t^2)+\frac{w^2}{2(1-w)^2}+ {\frac{A^1(t^2)}{2[1-A^1(t^2)]}+\frac{t^2A^{1'}(t^2)}{[1-A^1(t^2)]^2}}.
\end{align*}
We approximate $A^{1'}(r^2)$ by inserting the estimate of $r$ into $\big[{\sum_{n=1}^NA^1_n(r^2)^n-\sum_{n=1}^NA_n^1(r^2-0.001)^n}\big]/{0.001}$. We get $\phi_{ww}(r,s) \approx 5.2993$, $\phi_t(r,s) \approx 1.6716$, and $A^{1'}(r^2) \approx 1.1308$. Finally, as $n \rightarrow \infty$, $A^1_n \sim [\delta/(2\sqrt{\pi})]n^{-\frac{3}{2}}r^{-n}$, with $r \approx 0.2344$ and $\delta \approx 0.3846$.

\section{Leaf-labeled simplex time-consistent galled trees} \label{sec:ll1tcgt}

We now move to leaf-labeled simplex time-consistent galled trees. As we observed for general galled trees, the leaf-labeled case is similar to the unlabeled case.

\subsection{No galls}
\label{sec:1compnogallslabeled}

Leaf-labeled simplex time-consistent galled trees with no galls are the same as leaf-labeled time-consistent galled trees with no galls and general galled trees with no galls: all are simply the leaf-labeled rooted binary trees covered in Section \ref{sec:nogalls}. 

\subsection{One gall}

In the case of one gall, using the symbolic method as in Section \ref{sec:1tcgt1g} but for labeled classes, we obtain  
\begin{equation}
     \mathfrak{E}_1^1 = \{\circ\} \times \Big[\underbrace{\mathfrak{U} \star \mathfrak{E}_1^1}_1 
    \textrm{  }\dot{\cup}\textrm{  }
    \underbrace{\{\square\} \star \text{SET}_2 \big(\text{SEQ}^+(\mathfrak{U}) \big)}_2 \Big].
\end{equation}
The exponential generating function 
$\mathfrak{E}_1^1(t)=\sum_{n=0}^\infty (e^1_{n,1}/n!)t^n$ is
\begin{align}
    \mathfrak{E}_1^1(t) &= \mathfrak{U}(t) \, \mathfrak{E}_1^1(t) + \frac{t}{2}\bigg(\frac{\mathfrak{U}(t)}{1 - \mathfrak{U}(t)}\bigg)^2 \nonumber \\
     &= \frac{t \, \mathfrak{U}(t)^2}{2[1 - \mathfrak{U}(t)]^3}.
\end{align}

For the limit, we have, as $t\rightarrow\frac{1}{2}^{-}$, 
\begin{equation}
\mathfrak{E}_1^1(t)\sim \frac{1}{4} (1-2t)^{-\frac{3}{2}}.
\label{eq:E1-labeled}
\end{equation}
For the asymptotic growth of $e_{n,1}^1$, similarly to the unlabeled case, as $n \rightarrow \infty$, by eq.~(\ref{eq:transferth}), $\gamma(\frac{3}{2})=\sqrt{\pi}/2$, and the Stirling approximation, we have
\begin{equation}
    e_{n,1}^1 = [t^n] \mathfrak{E}^1_1(t) \, n! \sim \frac{1}{4}  \frac{n^{\frac{1}{2}}}{\Gamma(\frac{3}{2})} \Big( \frac{1}{2} \Big)^{-n} n! 
    \sim \frac{\sqrt{2}}{2} n^{n+1} \Big( \frac{2}{e} \Big)^n.
\end{equation}
Comparing to eq.~(\ref{eq:en1asy}), we have $e_{n,1}^1 \sim \frac{1}{2} e_{n,1} \sim \frac{1}{2} \tilde{e}_{n,1}$. The first terms of $e_{n,1}^1$ are calculated in Section \ref{sec:exactll1tcgt}. They appear in Table \ref{table:e1ng}.

\subsection{Two galls}

For two galls, from the symbolic method, similarly to the unlabeled case (Section \ref{sec:1tcgt2g}), we get
\begin{equation}
 \mathfrak{E}_2^1 = \{\circ\} \times  
 \underbrace{\Big[\mathfrak{U}\star \mathfrak{E}_2^1}_1 
    \textrm{  }\dot{\cup}\textrm{  } 
    \underbrace{\text{SET}_2(\mathfrak{E}_1^1)}_2
    \textrm{  }\dot{\cup}\textrm{  } 
    \underbrace{\{\square\}\star \big(\text{SEQ}(\mathfrak{U}) \star \mathfrak{E}_1^1 \star \text{SEQ}(\mathfrak{U})\big) \star \text{SEQ}^{+}(\mathfrak{U})}_3 \Big].
\end{equation}
The exponential generating function $\mathfrak{E}_2^1(t)=\sum_{n=0}^\infty (e^1_{n,2}/n!)t^n$ then satisfies
\begin{align}
    \mathfrak{E}_2^1(t) &= \mathfrak{U}(t) \, \mathfrak{E}_2^1(t) + \frac{1}{2}\mathfrak{E}_1^1(t)^2  + t\left(\frac{\mathfrak{E}_1^1(t)}{[1 - \mathfrak{U}(t)]^2}\right)\left(\frac{\mathfrak{U}(t)}{1 - \mathfrak{U}(t)}\right) \nonumber \\
    &= \frac{\mathfrak{E}_1^1(t)^2}{2[1 - \mathfrak{U}(t)]}  + \frac{t \,  \mathfrak{E}_1^1(t) \, \mathfrak{U}(t)}{[1 - \mathfrak{U}(t)]^4}.
\end{align}
As $t \rightarrow \frac{1}{2}^-$, following eqs.~(\ref{eq:U-labeled}) and (\ref{eq:E1-labeled}), we obtain 
\begin{equation} 
\mathfrak{E}_2^1(t) \sim \frac{5(\frac{1}{2})^2}{8(1-2t)^{\frac{7}{2}}}.
\end{equation}

Similarly to the unlabeled case, as $n \rightarrow \infty$, eq.~(\ref{eq:transferth}), $\Gamma{(\frac{7}{2})} = \frac{15}{8}\sqrt{\pi}$, and the Stirling approximation yield
\begin{equation}
    e_{n,2}^1 = [t^n] \mathfrak{E}_2^1(t) \, n! \sim 
    \frac{5 (\frac{1}{2})^2}{8} \frac{n^{\frac{5}{2}}(\frac{1}{2})^{-n}}{\Gamma(\frac{7}{2})} n! = \frac{1}{3\sqrt{\pi}}n^{\frac{5}{2}}\Big({\frac{1}{2}}\Big)^{-n+2} n! \sim \frac{1}{4} \frac{\sqrt{2}}{3} n^{n+3} \Big( \frac{2}{e} \Big)^n.
\end{equation}
Comparing to eq.~(\ref{eq:en2asy}), we have $e_{n,2}^1 \sim \frac{1}{4} e_{n,2} \sim \frac{1}{4} \tilde{e}_{n,1}$. The first terms of $e_{n,2}^1$ are calculated in Section \ref{sec:exactll1tcgt}. They appear in Table \ref{table:e1ng}.

\subsection{Bivariate generating function}

Next, to find univariate generating functions $\mathfrak{E}^1_g(t)$ for specific values of $g$, we examine the bivariate exponential generating function $\mathfrak{G}^1(t, u) = \sum_{n=0}^\infty \sum_{m=0}^\infty ({e_{n,m}^1}/{n!}) t^n u^m,$ where $e^1_{n,m}$ denotes the number of leaf-labeled simplex time-consistent galled trees with $n$ leaves and $m$ galls. 

For the symbolic method, the tree can just be one node (component 1 in eq.~(\ref{eq:g1tllsm}) and Figure \ref{fig:1tcgtbivariate}); it can have no root gall and two leaf-labeled simplex time-consistent galled trees descended from the root (component 2), or it can have a root gall and two nonempty sequences of leaf-labeled simplex time-consistent galled trees leading from the root to the reticulation node (component 3). The reticulation node itself has a single leaf as its subtree. In symbolic notation, we have 
\begin{equation}
    \mathfrak{G}^1 = \underbrace{\{\square\}}_{1} 
    \textrm{  }\dot{\cup}\textrm{  }  
    \{\circ\} \times \big[ \underbrace{\text{SET}_2(\mathfrak{G}^1)}_{2} 
    \textrm{  }\dot{\cup}\textrm{  }  
    \underbrace{\mu \times \{\square\} \star \text{SET}_2 \big( \text{SEQ}^+(\mathfrak{G}^1) \big)}_{3} \big].
    \label{eq:g1tllsm}
\end{equation}

Writing this expression as a generating function, we get $$\mathfrak{G}^1(t, u) = t + \frac{1}{2}\mathfrak{G}^1(t, u)^2 + \frac{ut}{2}\left(\frac{\mathfrak{G}^1(t, u)}{1 - \mathfrak{G}^1(t, u)}\right)^2.$$

\subsection{Fixed number of galls} 
\label{sec:fixedll1tcgt}

As in the unlabeled case in Section \ref{sec:fixedul1tcgt}, we derive the generating function $\mathfrak{E}_g^1(t)$ for leaf-labeled simplex time-consistent galled trees with exactly $g$ galls by noting that, for each $g \geq 1$, 
$$\mathfrak{E}_g^1(t) = \frac{1}{g!} \left(\frac{\partial^g}{\partial u^g} \mathfrak{G}^1\right)(t, 0).$$ 
We actually completed most of the steps in the unlabeled case, as the terms in the leaf-labeled case all appear in the bivariate generating function for the unlabeled case; the leaf-labeled case does not have different cases for even and odd $g$. We copy the appropriate terms from eq.~(\ref{eq:gf3}):

\pagebreak

\begin{align}
\mathfrak{E}_g^1(t) & = \frac{1}{2[1-\mathfrak{U}(t)]}\Bigg[ \bigg( \sum_{\ell=1}^{g-1} \mathfrak{E}^1_{\ell}(t) \, \mathfrak{E}^1_{g-\ell}(t) \bigg) \nonumber \\ 
& \quad + t \sum_{k_1 + 2k_2 + \ldots \atop + (g-1)k_{g-1} = g-1} 
\binom{k_1+k_2 + \ldots+k_{g-1}}{k_1,k_2, \ldots ,k_{g-1}}
\frac{2\mathfrak{U}(t) + (\sum_{\ell=1}^{g-1} k_{\ell}) - 1}{[1 - \mathfrak{U}(t)]^{(\sum_{\ell=1}^{g-1} k_{\ell}) + 2}} \prod_{\ell=1}^{g-1} \mathfrak{E}^1_{\ell}(t)^{k_{\ell}} \Bigg].
\end{align}

As the terms that determine the asymptotic approximation in the unlabeled case are already present in the generating function in the leaf-labeled case, we can immediately deduce the approximation as $t$ approaches the radius of convergence of the generating function, $t \rightarrow \frac{1}{2}^-$:
\begin{equation}
    \mathfrak{E}_g^1(t) \sim  \frac{C_{2g-1}}{2^{2g-1}} \frac{1}{2^g (1-2t)^{2g-\frac{1}{2}}} = \frac{(4g-3)!!}{(2g)! \, 2^g (1-2t)^{2g-\frac{1}{2}}}.
\end{equation}
As $n \rightarrow \infty$, by eq.~(\ref{eq:transferth}) and the Stirling approximation,
\begin{equation}
    e_{n,g}^1 \sim \frac{2^{2g-1}}{(2g)!\sqrt{\pi}}n^{2g-\frac{3}{2}}\Big(\frac{1}{2}\Big)^{-n+g}n!
    \sim \frac{2^{g-1}\sqrt{2}}{(2g)!}n^{n+2g-1}\Big(\frac{2}{e}\Big)^n. \label{eq:1identical}
\end{equation}

Values of $e_{n,g}^1$ for small $n$ and $g$ are calculated in Section \ref{sec:exactll1tcgt}, and they appear in Table \ref{table:e1ng}.

\subsection{Arbitrary number of galls} 
\label{sec:allll1tcgt}

We denote the exponential generating function of leaf-labeled simplex time-consistent galled trees with an arbitrary number of galls by $\mathfrak{A}^1(t) = \sum_{n=0}^\infty (a^1_{n}/n!) t^n$. The tree has three possibilities: (1) it is trivial with only one node (component 1 in eq.~(\ref{eq:a1tll}) and Figure \ref{fig:1tcgtbivariate}); (2) it has no root gall, and the root node branches into two leaf-labeled simplex time-consistent galled trees (component 2); (3) it has a root gall with two nonempty sequences of nodes from the root to the reticulation node, each with a leaf-labeled simplex time-consistent galled tree as a subtree, and a leaf descends from the reticulation node (component 3). In symbolic notation, 
\begin{equation}
    \mathfrak{A}^1 = \underbrace{\{\square\}}_{1} 
    \textrm{  }\dot{\cup}\textrm{  } 
    \{\circ\} \times \big[\underbrace{\text{SET}_2(\mathfrak{A}^1)}_{2} 
    \textrm{  }\dot{\cup}\textrm{  } 
    \underbrace{\{\square\} \star \text{SET}_2 \big(\text{SEQ}^+(\mathfrak{A}^1) \big)}_{2} \big].
    \label{eq:a1tll}
\end{equation}
The exponential generating function is 
\begin{equation}
    \mathfrak{A}^1(t) = t + \frac{1}{2}\mathfrak{A}^1(t)^2 + \frac{t}{2} \bigg(\frac{\mathfrak{A}^1(t)}{1 - \mathfrak{A}^1(t)}\bigg)^2.
    \label{eq:llA1tgf}
\end{equation}

Following Section \ref{sec:allul1tcgt}, we find the asymptotic approximation using the asymptotics of implicit tree-like classes theorem.
\begin{proposition}
    $\mathfrak{A}^1(t)$ belongs to the smooth implicit-function schema.
\end{proposition}

\begin{proof}
    To show that $\mathfrak{A}^1(t)$ belongs to the smooth implicit-function schema we prove the set of conditions set forth in Section \ref{sec:allul1tcgt}. (1) $\mathfrak{A}^1(t)$ is analytic at $0$ from the form of eq.~(\ref{eq:llA1tgf}). (2) $a_n^1(t) \geq 0$ for all $n \geq 0$, as a count of a class of objects. (3) $a_0^1 = 0$, as no simplex time-consistent galled-trees exist with no leaves.
\begin{enumerate}[ {(}1{)} ]    \setcounter{enumi}{3}

\item We define $\psi(t,w) = \sum_{n=0}^\infty \sum_{k = 0}^\infty s_{n,k} t^n w^k$ as 
\begin{equation*}
\psi(t,w) = \sum_{n=0}^\infty \sum_{k=0}^\infty s_{n,k} t^n w^k
= t + \frac{1}{2}w^2 + \frac{t}{2} \bigg(\frac{w}{1-w}\bigg)^2,
\end{equation*}
so that $\mathfrak{A}^1(t)=\psi \big(t,\mathfrak{A}^1(t) \big)$. We also have
\begin{enumerate}
\item $\psi$ is analytic around $(0,0)$ because it is a rational function in $w$ and $t$.
\item The $s_{n,k}$ are all nonnegative from the definition of $\psi$.
\item $s_{0,1}\neq 1$, as it is the coefficient of $w$, $s_{0,1}=0$.
\item $s_{0,2} > 0$, as it is the coefficient of $t^0w^2$, $s_{0,2}= \frac{1}{2}$.
\item $s_{0,0}=0$, because $\psi(0,0)=0$.
\item We demonstrate below that the characteristic system has a solution.
\end{enumerate}     
\end{enumerate}

The characteristic system is:
\begin{align}
    \omega &= \tau+\frac{1}{2}\omega^2+\frac{\tau}{2}\tau \bigg(\frac{\omega}{1-\omega}\bigg)^2, \label{eq:eq1implicitll} \\
    1 &= \omega+\tau\frac{\omega}{(1-\omega)^3}. \label{eq:eq2implicitll}
\end{align}
 
From eq.~(\ref{eq:eq2implicitll}), we get $\tau = {(1-\omega)^4}/{\omega}$, from which eq.~(\ref{eq:eq1implicitll}) gives us
\begin{align}
    \omega &= \frac{3-\sqrt{3}}{3} \approx 0.4226, \\
\label{eq:tau}
    \tau   &= \frac{3+\sqrt{3}}{18} \approx 0.2629.
\end{align}
Therefore, $\mathfrak{A}^1(t)$ belongs to the smooth implicit-function schema.
\end{proof}

\medskip
According to the \emph{asymptotics of implicit tree-like classes} theorem, as $n \rightarrow \infty$, ${a^1_n} \sim [{\delta}/({2\sqrt{\pi}})]n^{-\frac{3}{2}}\tau^{-n} n!$, where $\delta = \sqrt{{2\tau\psi_t(\tau,\omega)}/{\psi_{ww}(\tau,\omega)}}$. We have $\psi_t(t,w) = 1+{w^2}/[{2(1-w)^2}]$ and $\psi_{ww}(t,w) = 1+t(1+2w)/{(1-w)^4}$, so that by  Stirling's approximation,
\begin{equation}
    a^1_n \sim \frac{\delta}{2\sqrt{\pi}}n^{-\frac{3}{2}}\tau^{-n}n! \sim \frac{\delta}{\sqrt{2}}n^{n-1} \Big(\frac{\tau^{-1}}{e}\Big) ^{n}, 
    \label{eq:tauinv}
\end{equation}
with $\delta = (9-\sqrt{3}) \sqrt{3(9+\sqrt{3})}  /117 \approx 0.3525$ and $\tau = (3+\sqrt{3})/18 \approx 0.2629$ as in eq.~(\ref{eq:tau}).

The first values of $a^1_n$ are calculated in Section \ref{sec:exactll1tcgt} as the sum $\sum_{g=0}^{\lfloor \frac{n-1}{2} \rfloor}e^1_{n,g}$. They appear in Table {\ref{table:e1ng}}.

\section{Numerical computations by recursive enumeration} 
\label{sec:exact}

We now compute the exact numbers of phylogenetic networks in the four classes for which we have studied generating functions and asymptotics: unlabeled general galled trees (Section \ref{sec:ulggt}), leaf-labeled general galled trees (Section \ref{sec:llggt}), unlabeled simplex time-consistent galled trees (Section \ref{sec:ul1tcgt}), and leaf-labeled simplex time-consistent galled trees (Section \ref{sec:ll1tcgt}). We perform these computations using recursions.

\subsection{Unlabeled general galled trees, $\tilde{E}_{n,g}$} \label{sec:exactulggt}

In eqs.~(26) and (27) of \cite{AgranatTamirEtAl24a}, we calculated a recursion for $E_{n,g}$, the number of unlabeled time-consistent galled trees with $n$ leaves and $g$ galls (eqs. (26) and (27) there). The corresponding number of unlabeled general galled trees $\tilde{E}_{n,g}$ is the sum of $E_{n,g}$ and a term for the additional cases allowed with general galled trees but not time-consistent galled trees.  

We then have $\tilde{E}_{1,0}=1$, and for $n \geq 2$, $0 \leq g \leq {n-1}$,
\begin{align}
    \tilde{E}_{n,g} &=  \frac{1}{2}\bigg{[}\bigg{(}\sum_{\mathbf{c} \in C(n,2)}\sum_{\mathbf{d} \in C(g+2,2)}\prod_{i=1}^2\tilde{E}_{c_i,d_{i}-1}\bigg{)} 
    + (\delta_{0, n \textrm{ mod } 2})(\delta_{0, g \textrm{ mod } 2}) \tilde{E}_{\frac{n}{2},\frac{g}{2}} \nonumber \\
    & \quad + \bigg{(}\sum_{k=3}^{n}(k-2)\sum_{\mathbf{c}\in C(n,k)}\sum_{\mathbf{d}\in C(g-1+k,k)}\prod_{i=1}^k\tilde{E}_{c_i,d_{i}-1}\bigg{)} \nonumber\\
    & \quad + \bigg{(}\sum_{a=1}^{\lfloor \frac{n-1}{2} \rfloor}\sum_{\mathbf{c} \in C_p(n,2a+1)}\sum_{\mathbf{d} \in C_p(g-1+2a+1,2a+1)}\prod_{i=1}^{a+1} \tilde{E}_{c_i,d_{i}-1}\bigg{)}\bigg{]} \nonumber \\
    & \quad + \sum_{k=2}^n\sum_{\mathbf{c}\in C(n,k)}\sum_{\mathbf{d}\in C(g-1+k,k)}\prod_{i=1}^k\tilde{E}_{c_i,d_i-1}.
    \label{eq:Etildngrec}
\end{align}
In this expression, recall that $C(a,b)$ is the number of compositions of $a$ into $b$ parts, and $C_p(a,b)$ is the number of palindromic compositions. The last term is the term that is additional in relation to the time-consistent case; this term traverses possible numbers $k$ of non-root nodes in the root gall, distributing the $n$ leaves among their children according to a composition $\mathbf{c}$ in the set of compositions of $n$ into $k$ parts, $C(n,k)$. We distribute the remaining $g-1$ galls (one of the $g$ galls is the root gall) across these children according to a composition $\mathbf{d}$ in $C(g-1+k,k)$, where the $k$ added to $g-1$ is due to the technicality that compositions have positive parts, and it is possible for no galls to be descended from a non-root node in the root gall; the added $k$ is deducted in the $k$ appearances of $-1$ in expressions $\tilde{E}_{c_i,d_i-1}$. In the final term of eq.~(\ref{eq:Etildngrec}), the position of the reticulation node is specified because the term represents the case in which one of the sequences of nodes from the root of the root gall to the reticulation node is empty. 

The number of unlabeled general galled trees with $n$ leaves and any number of galls, $\tilde{A}_n$, is the sum of the numbers of unlabeled general galled trees with $n$ leaves and $g$ galls for all possible numbers of galls, or $\tilde{A}_n = \sum_{g=0}^{n-1}\tilde{E}_{n,g}$. The numerical values in Table \ref{table:Etildng} suggest a pattern in the numbers of general galled trees with the maximum number of galls.
\begin{proposition} \label{prop:enn-1cn-1}
The number of unlabeled general galled trees with $n$ leaves and the maximal number of galls, $n-1$, is a Catalan number, $\tilde{E}_{n,n-1} = C_{n-1}$.
\end{proposition}

\begin{proof}
We show a bijection between (non-plane) general galled trees with $n$ leaves and $n-1$ galls and plane rooted binary trees with $n$ leaves and $n-1$ internal nodes; the latter is counted by Catalan number $C_{n-1}$~\citep[Theorem 1.5.1(ii), p.~8]{Stanley2015}. First, a general galled tree has $n-1$ galls if and only if all galls have three nodes and all internal nodes are part of galls. The tree contains only 3-node galls and leaves. 

The galls are asymmetric: the two descendants of a top node are nodes of two different types, a hybridizing node and a reticulation node. This asymmetry in the gall corresponds to the distinction between left and right descendants of a tree node in an associated plane rooted binary tree. Because the general galled trees we consider are non-plane, we specify, without loss of generality, that the reticulation node is to the left of the top node in all galls. Therefore, a bijection exists between general galled trees with $n$ leaves and $n-1$ nodes and plane rooted binary trees with $n-1$ internal nodes, in that each gall can be uniquely assigned to an internal node in the binary tree.
\end{proof}

\medskip
Figure \ref{fig:enn-1cn-1} shows an example of the proposition in the case of $n=4$.

  \begin{figure}[tb]
    \centering
    \includegraphics[width=13cm]{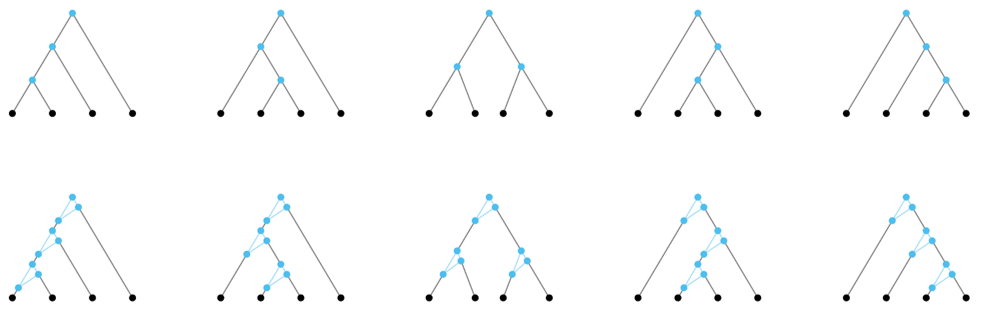}
    \caption{The bijection between plane rooted binary trees with 4 leaves (top) and non-plane general galled trees with 4 leaves and 3 galls (bottom). The bijection matches internal nodes in a plane rooted binary tree with the galls in the corresponding general galled tree (marked in light blue).}
    \label{fig:enn-1cn-1}
  \end{figure}
  
\subsection{Leaf-labeled general galled trees, $\tilde{e}_{n,g}$} \label{sec:exactllggt}

A recursion for $e_{n,g}$, the number of leaf-labeled time-consistent galled trees with $n$ leaves and $g$ galls appears in eq.~(56) of \cite{AgranatTamirEtAl25}. The number of general galled trees $\tilde{e}_{n,g}$ is obtained by adding a term to that recursion.

We have $\tilde{e}_{1,0}=1$, and for $n\geq 2$, $0 \leq g\leq n-1 $,
\begin{align}
    \tilde{e}_{n,g} &= \frac{1}{2}\bigg{[}\bigg(\sum_{m=1}^{n-1} \binom{n}{m} \sum_{\ell=0}^g \tilde{e}_{m,\ell}\tilde{e}_{n-m,g-\ell}\bigg) \nonumber \\
    & + \bigg(\sum_{k=3}^{n}(k-2)\sum_{\mathbf{c}\in C(n,k)}\sum_{\mathbf{d}\in C(g-1+k,k)}\binom{n}{c_1,c_2,\ldots,c_k}\prod_{i=1}^k \tilde{e}_{c_i,d_{i}-1}\bigg)\bigg{]} \nonumber \\
    & + \bigg(\sum_{k=2}^{n}\sum_{\mathbf{c}\in C(n,k)}\sum_{\mathbf{d}\in C(g-1+k,k)}\binom{n}{c_1,c_2,\ldots,c_k}\prod_{i=1}^k \tilde{e}_{c_i,d_{i}-1}\bigg).
    \label{eq:etildngrec}
\end{align}
The extra term is identical to that in the unlabeled case in eq.~(\ref{eq:Etildngrec}), with the addition of a multinomial coefficient to account for all possible ways in which the $n$ leaf labels can be distributed across the $k$ subtrees descended from non-root nodes in the root gall.

The number of leaf-labeled general galled trees with $n$ leaves and any number of galls, $\tilde{a}_n$, sums leaf-labeled general galled trees with $n$ leaves and $g$ galls for all possible numbers of galls, or $\tilde{a}_n = \sum_{g=0}^{n-1}\tilde{e}_{n,g}$. We obtain a corollary of Proposition \ref{prop:enn-1cn-1} for the leaf-labeled case, verifying a pattern visible in Table \ref{table:etildng}.
\begin{corollary}
The number of leaf-labeled general galled trees with $n$ leaves and the maximal number of galls, $n-1$, satisfies $\tilde{e}_{n,n-1} = C_{n-1} n! = (2n-2)!/(n-1)!$. 
\end{corollary}

\begin{proof}
The result adds leaf labels to the bijectively associated plane rooted binary trees in Proposition \ref{prop:enn-1cn-1}; the leaf labels are inherited by the associated general galled trees. The number of leaf-labeled plane rooted binary trees with $n$ leaves is $C_{n-1} n!$, so that $\tilde{e}_{n,n-1} = \tilde{E}_{n,n-1}n!$.
\end{proof}

\subsection{Unlabeled simplex time-consistent galled trees, $E^1_{n,g}$} \label{sec:exactul1tcgt}

For simplex time-consistent galled trees, the recursion is similar to that of time-consistent galled trees (eq.~(\ref{eq:Etildngrec}). However, because exactly one leaf and zero galls descend from reticulation nodes, $n-1$ leaves are distributed across $k-1$ rather than $k$ subtrees descended from a root gall. 

The simplex condition permits symmetry in a root gall only in trees with odd numbers of leaves and odd numbers of galls: one leaf is descended from the reticulation node, and one gall is the root gall. In this case, each side of the reticulation node has $\frac{n-1}{2}$ leaves and $\frac{g-1}{2}$ galls.

The base case of the recursion is $E^1_{1,0}=1$, and for $n\geq2$, $0 \leq g\leq \lfloor \frac{n-1}{2} \rfloor$,
\begin{align}
    {E}^1_{n,g} &=  \frac{1}{2}\bigg{[}\bigg{(}\sum_{\mathbf{c} \in C(n,2)}\sum_{\mathbf{d} \in C(g+2,2)}\prod_{i=1}^2{E}^1_{c_i,d_{i}-1}\bigg{)} 
    + (\delta_{0, n \textrm{ mod } 2})(\delta_{0, g \textrm{ mod } 2}) {E}^1_{\frac{n}{2},\frac{g}{2}} \nonumber \\
    & \quad + \bigg{(}\sum_{k=3}^{n}(k-2)\sum_{\mathbf{c}\in C(n-1,k-1)}\sum_{\mathbf{d}\in C\big(g-1+(k-1),k-1 \big)}\prod_{i=1}^{k-1}{E}^1_{c_i,d_{i}-1}\bigg{)} \nonumber\\
    & \quad + (\delta_{1, n \textrm{ mod } 2})(\delta_{1, g \textrm{ mod } 2})
    \bigg{(}\sum_{a=1}^{\frac{n-1}{2} }\sum_{\mathbf{c} \in C(\frac{n-1}{2},a)}\sum_{\mathbf{d} \in C(\frac{g-1}{2}+a,a)}\prod_{i=1}^{a} {E}^1_{c_i,d_{i}-1}\bigg{)}\bigg{]}.
    \label{eq:E1ngrec}
\end{align}

The number of unlabeled simplex time-consistent galled trees with $n$ leaves and any number of galls, $A^1_n$, is the sum of the numbers of unlabeled simplex time-consistent galled trees with $n$ leaves and $g$ galls for all possible numbers of galls ($0,1,...,\lfloor (n-1)/2 \rfloor$), or $A^1_n = \sum_{g=0}^{\lfloor \frac{n-1}{2} \rfloor}E^1_{n,g}$.

We observe a pattern in the numerical values in Table \ref{table:E1ng}, and verify it in the following proposition.
\begin{proposition} \label{prop:enn-1/2un+1/2}
The number of unlabeled simplex time-consistent galled trees with $n$ leaves, $n$ odd, and the maximal number of galls, $\frac{n-1}{2}$, is equal to the number of non-plane unlabeled rooted binary trees with $\frac{n+1}{2}$ leaves, ${E}^1_{n,(n-1)/2}=U_{(n+1)/2}$.
\end{proposition}

\begin{proof}
We produce a bijection between unlabeled simplex time-consistent galled trees with the maximal number of galls, $\frac{n-1}{2}$, and unlabeled rooted binary trees. A simplex time-consistent galled tree with the maximal number of galls has $\frac{n-1}{2}$ galls, each with a top node, two hybridizing nodes, and a reticulation node from which a leaf descends. The bijection is between the $\frac{n-1}{2}$ branching events in a non-plane unlabeled rooted binary tree (with $\frac{n+1}{2}$ leaves) and the galls in the unlabeled $n$-leaf simplex time-consistent galled tree. Each of the $\frac{n-1}{2}$ branching events in the binary tree becomes a 4-node gall by adding two hybridizing nodes on the two edges descending from the tree node and a reticulation node descended from these hybridizing nodes; a leaf descends from the reticulation. The total number of leaves is $\frac{n+1}{2}+\frac{n-1}{2}=n$. 
\end{proof}

\medskip
Figure \ref{fig:enn-1/2un+1/2} shows an example with $n=7$, $\frac{n-1}{2}=3$, and $\frac{n+1}{2}=4$.

\begin{figure}[tb]
    \centering
    \includegraphics[width=17cm]{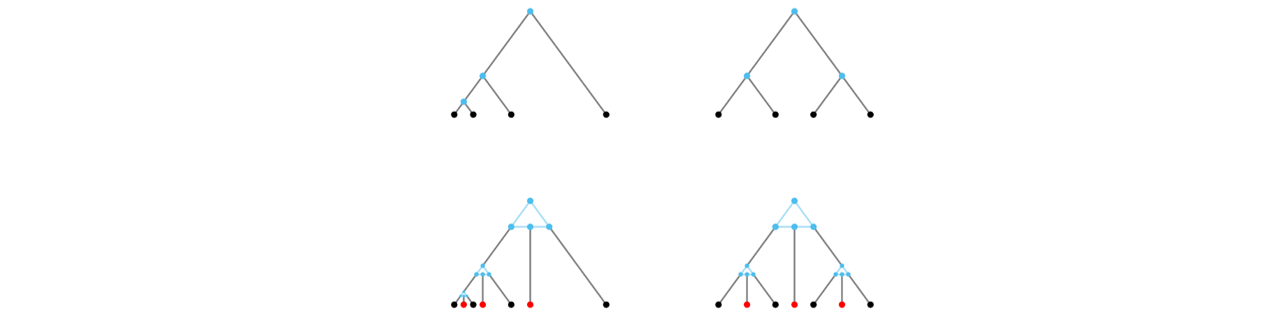}
    \caption{The bijection between non-plane rooted binary trees with 4 leaves (top) and non-plane simplex time-consistent galled trees with 7 leaves and 3 galls (bottom). The bijection matches internal nodes in a rooted binary tree with the galls in the corresponding simplex time-consistent galled tree (marked in light blue). The additional 3 leaves in the galled trees are marked in red.}
    \label{fig:enn-1/2un+1/2}
\end{figure}

\subsection{Leaf-labeled simplex time-consistent galled trees, $e^1_{n,g}$} \label{sec:exactll1tcgt}

In the leaf-labeled case, because each of the $n$ leaf labels is distinct, terms involving symmetric trees do not appear. We incorporate a factor of $n \binom{n-1}{c_1,c_2,\ldots,c_{k-1}}$ to place one label in the leaf descended from the reticulation node of the root gall and to distribute the other $n-1$ leaf labels among the children of the non-root non-reticulation nodes in the root gall.

Therefore, $e^1_{1,0} = 1$, and for $n \geq 2$, $0 \leq g\leq \lfloor \frac{n-1}{2} \rfloor$,
\begin{align}
     e^1_{n,g} &= \frac{1}{2}\bigg{[}\bigg(\sum_{m=1}^{n-1} \binom{n}{m} \sum_{\ell=0}^g e^1_{m,\ell}e^1_{n-m,g-\ell}\bigg) \nonumber \\
    & \quad + \bigg(\sum_{k=3}^{n}(k-2)\sum_{\mathbf{c}\in C(n-1,k-1)}\sum_{\mathbf{d}\in C\big(g-1+(k-1),k-1\big)} n\binom{n-1}{c_1,c_2,\ldots,c_{k-1}}\prod_{i=1}^{k-1} e^1_{c_i,d_{i}-1}\bigg)\bigg{]}.
    \label{eq:e1ngrec}
\end{align}

The number of leaf-labeled simplex time-consistent galled trees with $n$ leaves and any number of galls, $a^1_n$, is the sum of the numbers of leaf-labeled simplex time-consistent galled trees with $n$ leaves and $g$ galls for all possible numbers of galls, or $a^1_n = \sum_{g=0}^{\lfloor \frac{n-1}{2} \rfloor}e^1_{n,g}$.

We write $u_n = (2n-2)! / [2^{n-1} (n-1)!]$ for the number of leaf-labeled trees with $n$ specified labels.
\begin{corollary}
    The number of leaf-labeled simplex time-consistent galled trees with $n$ leaves, $n$ odd, and the maximal number of galls, $\frac{n-1}{2}$, satisfies $e^1_{n,(n-1)/2}=u_{(n+1)/2}\, {n!}/{(\frac{n+1}{2})!}$.
\end{corollary}
\begin{proof}
By Proposition \ref{prop:enn-1/2un+1/2}, a bijection exists between leaf-labeled trees with $\frac{n+1}{2}$ leaves, $n$ odd, and leaf-labeled simplex time-consistent galled trees with $n$ leaves and $\frac{n-1}{2}$ galls, with only the $\frac{n+1}{2}$ non-gall descendant leaves labeled. The number of ways of labeling those leaves is ${n \choose (n+1)/2}$. The labels of the $\frac{n-1}{2}$ gall descendant leaves can then be assigned in $(\frac{n-1}{2})!$ ways. The number of leaf-labeled simplex time-consistent galled trees with $n$ leaves, $n$ odd, and $\frac{n-1}{2}$ galls is 
$u_{(n+1)/2} {n \choose (n+1)/2} (\frac{n-1}{2})! = u_{(n+1)/2}\, {n!}/{(\frac{n+1}{2})!}$.
\end{proof}

\section{Discussion}

We have studied the combinatorial properties of two types of galled trees, general galled trees and simplex time-consistent galled trees, augmenting our previous studies of a third type of galled trees, the time-consistent galled trees. Our analysis has considered both unlabeled and leaf-labeled cases, and for each class of unlabeled or leaf-labeled galled trees, we have investigated the generating functions and asymptotic approximations for the number of galled trees with $n$ leaves and a fixed number of galls $g$ (Sections \ref{sec:ggt0galls}-\ref{sec:fixedulggt}, \ref{sec:nogalls}-\ref{sec:fixedllggt}, \ref{sec:1compnogalls}-\ref{sec:fixedul1tcgt}, and \ref{sec:1compnogallslabeled}-\ref{sec:fixedll1tcgt}), and the generating functions and asymptotics for the number of galled trees with $n$ leaves, summing across all possible numbers of galls (Sections \ref{sec:allulggt}, \ref{sec:allllggt}, \ref{sec:allul1tcgt}, and \ref{sec:allll1tcgt}). In Section \ref{sec:exact}, we examined numerical values for the enumerations for small values of $n$ and $g$. 

For the number of galled trees with $n$ leaves and $g$ galls, as $n \rightarrow \infty$, in both the unlabeled and leaf-labeled cases, the class of general galled trees ($\tilde{E}_{n,g}$) has the same asymptotic approximation as the time-consistent galled trees ($E_{n,g}$). The simplex time-consistent galled trees ($E^1_{n,g}$) are asymptotically less numerous in the subexponential term. In the unlabeled case, by
eqs.~(\ref{eq:Eng}) and (\ref{eq:Eng1}), for $g \geq 1$, we have 
\begin{equation}
E^1_{n,g} \rho^{-g} \sim E_{n,g} \sim \tilde{E}_{n,g} \sim \frac{2^{2g-1}}{(2g)! \, \gamma^{4g-1} \sqrt{\pi} } n^{2g-\frac{3}{2}} \rho^{-n},
\end{equation}
where $\rho \approx 0.40270$ and $\gamma \approx 1.13003$ are constants that trace originally to the asymptotic analysis of unlabeled trees with no galls, the asymptotic analysis of the Wedderburn--Etherington numbers (eq.~(\ref{eq:Wedderburn})). 

In the leaf-labeled case, many classes of phylogenetic networks have been studied. Using eq.~(\ref{eq:1}), (\ref{eq:2}), (\ref{eq:identical}), and (\ref{eq:1identical}), and switching from $g$ to $k$ to count reticulations in a more general sense, for $k \geq 1$, we have:  
\begin{align}
e^1_{n,k} \frac{(2k)!}{k!} & \sim e_{n,k} \frac{(2k)!}{2^k k!} \sim \tilde{e}_{n,k} \frac{(2k)!}{2^k k!} \nonumber \\
& \sim
PN_{n,k}\sim RV_{n,k} \sim TC_{n,k} \sim N_{n,k} \sim GN_{n,k} \sim GTC_{n,k} \sim \frac{2^{k-1}\sqrt{2}}{k!}n^{n+2k-1}\Big( \frac{2}{e} \Big)^n.
\end{align}
The imposition of a requirement that all reticulations are galls reduces the asymptotic count of the number of general galled trees below that of larger network classes. The simplex condition further reduces this count, as it prevents one subtree from having a further recursively defined structure, fixing that subtree as a leaf.

For the total number of galled trees, not restricting the number of galls, general galled trees possess an extra type of structure not present in time-consistent galled trees: in entry 4 of Figure \ref{fig:ggtbv}, a top node of a gall can be a parent of the reticulation node. As a result, galls can possess only two nodes other than the root, rather than a minimum of three, and the total number of galls can be as large as $n-1$ rather than $\lfloor \frac{n-1}{2} \rfloor$. Although asymptotic approximations of the numbers of general galled trees, time-consistent galled trees, and simplex time-consistent galled trees with fixed numbers of galls differ only in their subexponential growth in both the unlabeled and leaf-labeled cases, when considering all possible numbers of galls, the exponent also differs. In the unlabeled case, all grow with $[\delta/(2\sqrt{\pi})] n^{-\frac{3}{2}} \alpha^{-n}$ for constants $\alpha$ and $\delta$. The growth is fastest for general galled trees; the approximations for the constants $(\delta,\alpha)$ are $(0.1966, 0.1165)$ (eq.~(\ref{eq:fg})); for time-consistent galled trees, $(0.2762, 0.2073)$ (eq.~(42) of \cite{AgranatTamirEtAl24a}); for simplex time-consistent galled trees, $(0.3846, 0.2344)$ (Section \ref{sec:allul1tcgt}). In the leaf-labeled case, the growth follows $(\delta/\sqrt{2})n^{n-1} (\alpha^{-1}/{e})^n$, where $(\delta,\alpha)$ are approximately $(0.1894,0.1250)$ for general galled trees (eq.~(\ref{eq:sqrt17})), $(0.2548,0.2383)$ for time-consistent galled trees \citep[eq.~(25)]{FuchsAndGittenberger25}, and  $(0.3525,0.2629)$ for simplex time-consistent galled trees (eq.~(\ref{eq:tauinv})).

We have reported in Tables \ref{table:Etildng}-\ref{table:e1ng} recursive calculations of the exact numbers of unlabeled and leaf-labeled general galled trees and simplex time-consistent galled trees with specific numbers of leaves $n$ and galls $g$ for small numbers of leaves; these tables can be examined together with Tables 4 and 5 of \cite{AgranatTamirEtAl25}, which reports corresponding quantities for time-consistent galled trees. Comparing Tables \ref{table:etildng} and \ref{table:e1ng} to Tables \ref{table:Etildng} and \ref{table:E1ng}, the tables illustrate the much larger values for leaf-labeled than for unlabeled quantities, comparing Tables \ref{table:Etildng} and \ref{table:etildng} to Tables \ref{table:E1ng} and \ref{table:e1ng}, they illustrate the faster growth of the general galled trees. 

In the case that general galled trees are ``saturated'' with the maximal number of galls, $n-1$, we found in Proposition \ref{prop:enn-1cn-1} that the number of unlabeled general galled trees with $n$ leaves is the Catalan number $C_{n-1}$. For odd $n$, we found in Proposition \ref{prop:enn-1/2un+1/2} that the number of simplex time-consistent galled trees with the maximal number of galls $\lfloor \frac{n-1}{2} \rfloor$ is simply the number of unlabeled binary rooted trees $U_{(n+1)/2}$. These results augment corresponding results for other classes of phylogenetic networks with the maximal number of reticulation nodes, including Proposition 2 for tree-child networks in \cite{FuchsEtAl21b}, Lemma 2 for galled networks in \cite{FuchsEtAl22a}, Theorem 5 for reticulation-visible networks in \cite{ChangAndFuchs24}, and Theorem 22 for galled tree-child networks in \cite{ChangEtAl24}. Further numerical investigation of the counts of phylogenetic networks may uncover additional bijective constructions involving galled trees or other classes of networks with classic combinatorial structures.

\bigskip
\noindent {\small {\bf Acknowledgments.} Lily Agranat-Tamir, Karthik V.~Seetharaman, and Noah A.~Rosenberg acknowledge support from National Science Foundation grant DMS-2450005. Michael Fuchs was partially supported by the National
Science and Technology Council (NSTC), Taiwan under grant NSTC-113-2115-M-004-004-MY3.}

{\small \bibliography{refs}}

@article{GUNAWAN2020644,
title = {Counting and enumerating galled networks},
journal = {Discr. Appl. Math.},
volume = {283},
pages = {644-654},
year = {2020},
author = {A. D. M. Gunawan and J. Rathin and L. Zhang}
}

@article {ChangAndFuchs24,
    AUTHOR = {Y.-S. Chang and M. Fuchs},
     TITLE = {Counting phylogenetic networks with few reticulation
              vertices: galled and reticulation-visible networks},
   JOURNAL = {Bull. Math. Biol.},
  FJOURNAL = {Bulletin of Mathematical Biology. A Journal Devoted to
              Research at the Interface of the Life and Mathematical
              Sciences},
    VOLUME = {86},
      YEAR = {2024},
     PAGES = {76},
   MRCLASS = {99-06},
  MRNUMBER = {4747507},
}

@incollection{ChangEtAl24,
   AUTHOR = {Y.-S. Chang and M. Fuchs and G.-R. Yu},
     TITLE = {Galled tree-child networks},
 BOOKTITLE = {35th {I}nternational {C}onference on {P}robabilistic,
              {C}ombinatorial and {A}symptotic {M}ethods for the {A}nalysis
              of {A}lgorithms},
    SERIES = {Leibniz International Proceedings in Informatics, LIPIcs},
    VOLUME = {302, Article 2},
 PUBLISHER = {Schloss Dagstuhl---Leibniz-Zentrum f\"{u}r Informatik, Wadern},
      YEAR = {2024}

}

@Article{KongEtAl22,
  author =   "S. Kong and J. C. Pons and L. Kubatko and K. Wicke",
  title =    "Classes of explicit phylogenetic networks and their biological and mathematical significance",
  journal =  "J. Math. Biol.",
  year =     2022,
  volume =   "84",
  pages =    "47",
  key =      ""
}

@incollection{AgranatTamirEtAl24b,
   AUTHOR = {L. Agranat-Tamir and M. Fuchs and B. Gittenberger and N. A. Rosenberg},
     TITLE = {Asymptotic enumeration of rooted binary unlabeled galled trees with a fixed number of galls},
 BOOKTITLE = {35th {I}nternational {C}onference on {P}robabilistic,
              {C}ombinatorial and {A}symptotic {M}ethods for the {A}nalysis
              of {A}lgorithms},
    SERIES = {Leibniz International Proceedings in Informatics, LIPIcs},
    VOLUME = {302, Article 27},
 PUBLISHER = {Schloss Dagstuhl---Leibniz-Zentrum f\"{u}r Informatik, Wadern},
      YEAR = {2024}

}

@Article{AgranatTamirEtAl24a,
  author =   "L. Agranat-Tamir and S. Mathur and N. A. Rosenberg",
  title =    "Enumeration of Rooted Binary Unlabeled Galled Trees",
  journal =  "Bull. Math. Biol.",
  year =     2024,
  volume =   86,
  pages =    "45",
  key =      ""
}

@Article{Mansouri22,
  author =   "M. Mansouri",
  title =    "Counting general phylogenetic networks",
  journal =  "Australas. J. Combin.",
  year =     2022,
  volume =   "83",
  pages =    "40-86",
  key =  "",
  doi = ""
}

@Article{MeirAndMoon89a,
  author =   "A. Meir and J. W. Moon",
  title =    "On an asymptotic method in enumeration",
  journal =  "J. Comb. Theory, Ser. A",
  year =     1989,
  volume =   51,
  pages =    "77-89",
  key =      ""
}

@Article{MeirAndMoon89b,
  author =   "A. Meir and J. W. Moon",
  title =    "Erratum: ``{On} an asymptotic method in enumeration''",
  journal =  "J. Comb. Theory, Ser. A",
  year =     1989,
  volume =   52,
  pages =    "163",
  key =      ""
}

@Article{BienvenuEtAl21,
  author =   "F. Bienvenu and A. Lambert and M. Steel",
  title =    "Combinatorial and stochastic properties of ranked tree-child networks",
  journal =  "Random Struct. Alg.",
  year =     2022,
  volume =   60,
  pages =    "653-689",
  key =      ""
}

@Article{BouvelEtAl20,
  author =   "M. Bouvel and P. Gambette and M. Mansouri",
  title =    "Counting phylogenetic networks of level 1 and 2",
  journal =  "J. Math. Biol.",
  year =     2020,
  volume =   81,
  pages =    "1357-1395",
  key =      ""
}

@Article{FuchsEtAl19,
  author =   "M. Fuchs and B. Gittenberger and M. Mansouri",
  title =    "Counting phylogenetic networks with few reticulation vertices: tree-child and normal networks",
  journal =  "Australas J. Combin.",
  year =     2019,
  volume =   "73",
  pages =    "385-423",
  key =      ""
}

@Article{FuchsEtAl21,
  author =   "M. Fuchs and B. Gittenberger and M. Mansouri",
  title =    "Counting phylogenetic networks with few reticulation vertices: exact enumeration and corrections",
  journal =  "Australas J. Combin.",
  year =     {2021},
  volume =   89,
  pages =    "257-282",
  key =      ""
}

@Article{FuchsEtAl21b,
title = {On the asymptotic growth of the number of tree-child networks},
journal = {Eur. J. Combin.},
volume = {93},
pages = {103278},
year = {2021},
author = {M. Fuchs and G.-R. Yu and Louxin Zhang},
}

@Article{FuchsEtAl22a,
  author =   "M. Fuchs and G.-R. Yu and L. Zhang",
  title =    "Asymptotic enumeration and distributional properties of galled networks",
  journal =  "J. Comb. Theory Ser. A",
  year =     {2022},
  volume =   189,
  pages =    "105599",
  key =      ""
}

@Article{FuchsEtAl22b,
  author =   "M. Fuchs and E.-Y Huang and G.-R. Yu",
  title =    "Counting phylogenetic networks with few reticulation vertices: a second approach",
  journal =  "Discr. Appl. Math.",
  year =     {2022},
  volume =   320,
  pages =    "140-149",
  key =      ""
}

@article {FuchsEtAl25,
    AUTHOR = {Fuchs, M. and Steel, M. and Zhang, Q.},
     TITLE = {Asymptotic enumeration of normal and hybridization networks via tree decoration},
   JOURNAL = {Bull. Math. Biol.},
    VOLUME = {87},
      YEAR = {2025},
     PAGES = {69},
   MRCLASS = {92D15 (05C90 60C05 92B10)},
  MRNUMBER = {4902834},
}

@Article{CardonaAndZhang20,
  author =   "G. Cardona and L. Zhang",
  title =    "Counting and enumerating tree-child networks and their subclasses",
  journal =  "J. Comp. System Sci.",
  year =     2020,
  volume =   114,
  pages =    "84-104",
  key =      ""
}

@Article{Harding71,
  author =   "E. F. Harding",
  title =    "The probabilities of rooted tree-shapes generated by random bifurcation",
  journal =  "Adv. Appl. Prob.",
  year =     1971,
  volume =   3,
  pages =    "44-77",
  key =      ""
}

@Article{MathurAndRosenberg23,
  author =   "S. Mathur and N. A. Rosenberg",
  title =    "All galls are divided into three or more parts: recursive enumeration of labeled histories for galled trees",
  journal =  "Alg. Mol. Biol.",
  year =     2023,
  volume =   "18",
  pages =    "1",
  key =  ""
}

@Article{Otter48,
  author =   "R. Otter",
  title =    "The number of trees",
  journal =  "Ann. Math.",
  year =     1948,
  volume =  49,
  pages =    "583-599",
  key =      ""
}

@Book{FlajoletAndSedgewick09,
  author =   "P. Flajolet and R. Sedgewick",
  title =    "Analytic Combinatorics",
  publisher =    "Cambridge University Press",
  year =     2009,
  key =      "",
  address =  "Cambridge, UK",
  edition =  ""
}

@article{Felsenstein78,
author = {Felsenstein, J.},
title = {The number of evolutionary trees},
journal = {Syst. Zool.},
year = {1978},
pages = {27-33},
volume = {27},
}

@inproceedings{EdwardsAndCSforza64, 
author = {A. W. F. Edwards and L. L. Cavalli-Sforza},
address={London}, 
title={Reconstruction of evolutionary trees}, 
booktitle={Phenetic and Phylogenetic Classification}, 
editor={V. H. Heywood and J. McNeil},
publisher={Systematics Association},
pages={67-76},
year={1964},
}

@article{FuchsAndGittenberger25,
author = {Fuchs, M. and Gittenberger, B.},
year = {2025},
journal={J. Math. Biol.},
volume = {90},
number = {},
pages = {42},
title = {Sackin indices for labeled and unlabeled classes of galled trees},
}

@article{Zhang19,
author = {Zhang, L.},
year = {2019},
pages = {642},
title = {Generating normal networks via leaf insertion and nearest neighbor interchange},
volume = {20},
journal = {BMC Bioinformatics},
}

@article{AgranatTamirEtAl25,
author = {L. Agranat-Tamir and M. Fuchs and B. Gittenberger and N. A. Rosenberg},
year = {2025},
journal={arXiv.2504.16302},
volume = {},
pages = {},
title = {Enumerative combinatorics of unlabeled and labeled time-consistent galled trees},
}

@InProceedings{Zhang2022,
author="Zhang, L.",
editor="Jin, L. and Durand, D.",
title="The {Sackin} Index of Simplex Networks",
booktitle="Comparative Genomics: 19th International Conference, RECOMB-CG 2022, La Jolla, CA, USA, May 20–21, 2022, Proceedings",
year="2022",
publisher="Springer",
address="Berlin",
pages="52--67",
}

@book{Stanley2015, 
address={Cambridge}, 
title={Catalan Numbers}, 
publisher={Cambridge University Press}, 
author={Stanley, R. P.}, 
year={2015},
}

@article {FuchsANDSteel25,
    AUTHOR = {Fuchs, M. and Steel, M.},
     TITLE = {A dichotomy law for certain classes of phylogenetic networks},
   JOURNAL = {Bull. Math. Biol.},
    VOLUME = {87},
      YEAR = {2025},
     PAGES = {153},
   MRCLASS = {92D15 (05C90 92C42)},
  MRNUMBER = {4965827},
}

\begin{landscape}
\begin{table}[tb]
    \centering
    \begin{adjustbox}{width=1.38\textwidth}
    \begin{tabular}{|c|r|rrrrrrrrrrrr|}
    \hline
Number of    & \multicolumn{1}{c|}{Total number}       & \multicolumn{12}{c|}{Number of trees with a fixed number of galls ($\tilde{E}_{n,g}$)} \\ 
leaves ($n$) & \multicolumn{1}{c|}{of trees ($\tilde{A}_n$)} & $g=0$ & $g=1$ & $g=2$ & $g=3$ & $g=4$ & $g=5$ & $g=6$ & $g=7$ & $g=8$ & $g=9$ & $g=10$ & $g=11$ \\ \hline
 1 &           1 &   1 &      - &             - &              - &            - &          - &          - &          - &          - &         - &       - &      - \\
 2 &           2 &   1 &      1 &             - &              - &            - &          - &          - &          - &          - &         - &       - &      - \\
 3 &           8 &   1 &      5 &             2 &              - &            - &          - &          - &          - &          - &         - &       - &      - \\
 4 &          43 &   2 &     16 &            20 &              5 &            - &          - &          - &          - &          - &         - &       - &      - \\
 5 &         255 &   3 &     49 &           113 &             76 &           14 &          - &          - &          - &          - &         - &       - &      - \\
 6 &       1,637 &   6 &    140 &           526 &            634 &          289 &         42 &          - &          - &          - &         - &       - &      - \\
 7 &      11,004 &  11 &    392 &         2,143 &          4,030 &        3,198 &      1,098 &        132 &          - &          - &         - &       - &      - \\
 8 &      76,634 &  23 &  1,072 &         8,076 &         21,604 &       26,024 &     15,217 &      4,189 &        429 &          - &         - &       - &      - \\
 9 &     547,539 &  46 &  2,898 &        28,667 &        103,267 &      173,886 &    151,509 &     69,808 &     16,028 &      1,430 &         - &       - &      - \\
10 &   3,992,150 &  98 &  7,744 &        97,498 &        453,712 &    1,013,118 &  1,217,115 &    824,004 &    312,468 &     61,531 &     4,862 &       - &      - \\ 
11 &  29,579,670 & 207 & 20,538 &       320,388 &      1,869,315 &    5,328,525 &  8,392,643 &  7,749,108 &  4,271,018 &  1,374,266 &   236,866 &  16,796 &      - \\
12 & 222,089,534 & 451 & 54,095 &     1,024,750 &	   7,318,842 &   25,886,758 & 51,562,707 & 61,852,619 & 46,085,573 & 21,365,251 & 5,965,532 & 914,170 &	58,786 \\
\hline
\end{tabular}
\end{adjustbox}
\caption{Numbers $\tilde{E}_{n,g}$ of small unlabeled general galled trees with specified numbers of leaves ($n$) and galls ($g$). Entries $\tilde{E}_{n,g}$ are computed recursively from eq.~(\ref{eq:Etildngrec}) in Section \ref{sec:exactulggt}. For $g=0$, $E_{n,0}=U_n$.} 
\label{table:Etildng}
\end{table}

\begin{table}[tb]
    \centering
    \begin{adjustbox}{width=1.38\textwidth}
    \begin{tabular}{|c|r|rrrrrrrrrrrr|}
    \hline
Number of & \multicolumn{1}{c|}{Total number} & \multicolumn{12}{c|}{Number of trees with a fixed number of galls ($\tilde{e}_{n,g}$)} \\ 
leaves ($n$) & \multicolumn{1}{c|}{of trees ($\tilde{a}_n$)} & $g=0$ & $g=1$ & $g=2$ & $g=3$ & $g=4$ & $g=5$ & $g=6$ & $g=7$ & $g=8$ & $g=9$ & $g=10$ & $g=11$\\ \hline
 1 &                 1 &           1 &             - &             - &              - &                - &                 - &                 - &              - &               - &              - & - & -\\
 2 &                 3 &           1 &             2 &             - &              - &                - &                 - &                 - &              - &               - &              -& - & -\\
 3 &                36 &           3 &            21 &            12 &              - &                - &                 - &                 - &              - &               - &              - & - & -\\
 4 &               723 &          15 &           228 &           360 &            120 &                - &                 - &                 - &              - &               - &              - & - & -\\
 5 &            20,280 &         105 &         2,805 &          8550 &          7,140 &            1,680 &                 - &                 - &              - &               - &              - & - & -\\
 6 &           730,755 &         945 &        39,330 &       196,560 &        297,360 &          166,320 &            30,240 &                 - &              - &               - &              - & - & -\\ 
 7 &        32,171,580 &      10,395 &       623,385 &     4,639,320 &     11,007,360 &       10,735,200 &         4,490,640 &           665,280 &              - &               - &              - & - & -\\
 8 &     1,673,573,895 &     135,135 &    11,055,240 &   114,896,880 &    392,893,200 &      583,783,200 &       415,134,720 &       138,378,240 &     17,297,280 &               - &              - & - & -\\
 9 &   100,442,870,640 &   2,027,025 &   217,237,545 & 3,010,855,050 & 14,023,894,500 &   29,357,559,000 &    31,167,536,400 &    17,344,847,520 &  4,799,995,200 &     518,918,400 &              - & - & -\\
10 & 6,831,585,584,775 &  34,459,425 & 4,689,345,150 &83,706,777,000 &509,858,533,800 &1,427,560,734,600 & 2,100,224,926,800 & 1,719,046,929,600 &783,566,784,000 & 185,253,868,800 & 17,643,225,600 & - & -\\
11 & 519,288,366,989,700 & 654,729,075 &	110,367,613,125 &	2,469,131,626,500 &	19,067,559,310,800 &	68,780,493,381,600 &	133,494,142,181,400 &	149,527,115,337,600	& 99,101,998,195,200 &	38,188,761,811,200 &	7,877,700,230,400 &	670,442,572,800 &	- \\
12 & 43,626,178,967,384,480 & 13,749,310,575 &	2,813,814,441,900 &	77,183,193,087,000 &	737,459,213,891,400 &	3,330,980,344,692,000 &	8,233,128,718,214,400 &	12,036,684,871,411,200 &	10,810,119,006,086,400 &	6,002,304,743,635,200 &	2,001,271,079,808,000 &	366,061,644,748,800 &	28,158,588,057,600 \\
\hline
\end{tabular}
\end{adjustbox}
\caption{Numbers $\tilde{e}_{n,g}$ of small leaf-labeled general galled trees with specified numbers of leaves ($n$) and galls ($g$). Entries $\tilde{e}_{n,g}$ are computed recursively from eq.~(\ref{eq:etildngrec}) in Section \ref{sec:exactllggt}. For $g=0$, $\tilde{e}_{n,0}$ is the number of leaf-labeled rooted binary trees $(2n-3)!!=(2n-2)!/[2^{n-1} (n-1)!]$.}
\label{table:etildng}
\end{table}

\clearpage
\begin{table}[tb]
    \centering
    \begin{tabular}{|c|r|rrrrrrrr|}
    \hline
Number of & \multicolumn{1}{c|}{Total number} & \multicolumn{8}{c|}{Number of trees with a fixed number of galls ($E^1_{n,g}$)} \\ 
leaves ($n$) & \multicolumn{1}{c|}{of trees ($A^1_n$)} & $g=0$ & $g=1$ & $g=2$ & $g=3$ & $g=4$ & $g=5$ & $g=6$ & $g=7$\\ \hline
 1 &         1 &     1 &          - &           - &         - &         - &       - &     - & - \\
 2 &         1 &     1 &          - &           - &         - &         - &       - &     - & - \\
 3 &         2 &     1 &          1 &           - &         - &         - &       - &     - & - \\
 4 &         5 &     2 &          3 &           - &         - &         - &       - &     - & - \\
 5 &        15 &     3 &         11 &           1 &         - &         - &       - &     - & - \\
 6 &        47 &     6 &         32 &           9 &         - &         - &       - &     - & - \\
 7 &       158 &    11 &         96 &          49 &         2 &         - &       - &     - & - \\
 8 &       545 &    23 &        270 &         230 &        22 &         - &       - &     - & - \\
 9 &     1,940 &    46 &        758 &         951 &       182 &         3 &       - &     - & - \\
10 &     7,030 &    98 &      2,074 &       3,657 &     1,146 &        55 &       - &     - & - \\
11 &    25,917 &   207 &      5,638	&      13,242 &	    6,235 &	      589 &	      6 &     - & - \\
12 &    96,755 &   451 &	 15,132 & 	   45,955 &	   30,295 &  	4,789 &	    133 &     - & - \\
13 &   365,274 &   983 &  	 40,364 &	  153,943 &	  135,839 &	   32,331 &   1,803 &	 11 & - \\
14 & 1,391,629 & 2,179 &	106,858 &	  501,536 &	  571,309 &	  191,543 &	 17,881 &   323 & - \\
15 & 5,344,385 & 4,850 &	281,494 &	1,595,983 &	2,284,292 &	1,027,307 &	145,178 & 5,258 & 23 \\
\hline
\end{tabular}
\caption{Numbers $E^1_{n,g}$ of small unlabeled simplex time-consistent galled trees with specified numbers of leaves ($n$) and galls ($g$). Entries $E^1_{n,g}$ are computed recursively from eq.~(\ref{eq:E1ngrec}) in Section \ref{sec:exactul1tcgt}. The values of $\tilde{A}_n$ for $n\leq25$ missing from the table, as used in Section \ref{sec:allul1tcgt}, are calculated recursively using eq.~(\ref{eq:exactulAtildn}): $ A^1_{16}= 20,665,633, A^1_{17}= 80,394,281, A^1_{18}= 314,422,685, A^1_{19}= 1,235,563,322, A^1_{20}= 4,875,984,643, A^1_{21}= 19,316,320,501, A^1_{22}= 76,788,144,951, A^1_{23}= 306,222,020,733, A^1_{24}= 1,224,709,419,623, A^1_{25}= 4,911,122,651,176$.}
\label{table:E1ng}
\end{table}


\begin{table}[tb]
    \centering
    \begin{adjustbox}{width=1.38\textwidth}
    \begin{tabular}{|c|r|rrrrrrrr|}
    \hline
Number of & \multicolumn{1}{c|}{Total number} & \multicolumn{8}{c|}{Number of trees with a fixed number of galls ($e^1_{n,g}$)} \\ 
leaves ($n$) & \multicolumn{1}{c|}{of trees ($a^1_n$)} & $g=0$ & $g=1$ & $g=2$ & $g=3$ & $g=4$ & $g=5$ & $g=6$ & $g=7$ \\ \hline
 1 &              1 &          1 &             - &             - &              - &          - & - & - & - \\
 2 &              1 &          1 &             - &             - &              - &          - & - & - & - \\
 3 &              6 &          3 &             3 &             - &              - &          - & - & - & - \\
 4 &             63 &         15 &            48 &             - &              - &          - & - & - & - \\
 5 &            870 &        105 &           705 &            60 &              - &          - & - & - & - \\
 6 &         15,075 &        945 &        10,980 &         3,150 &              - &          - & - & - & - \\
 7 &        317,520 &     10,395 &       186,795 &       117,180 &          3,150 &          - & - & - & - \\
 8 &      7,891,695 &    135,135 &     3,487,680 &     3,916,080 &        352,800 &          - & - & - & - \\
 9 &    226,068,570 &  2,027,025 &    71,291,745 &   127,348,200 &     25,084,080 &    317,520 & - & - & - \\
10 &  7,332,968,475 & 34,459,425 & 1,587,996,900 & 4,173,488,550 &  1,475,107,200 & 61,916,400 & - & - & - \\ 
11 & 265,668,133,500  & 654,729,075 &	38,347,414,875 &	140,238,945,000 &	79,152,207,750 & 7,222,446,000 & 52,390,800 & - & - \\
12 & 10,633,098,774,375 & 13,749,310,575 &	998,905,446,000 &	4,875,076,206,000 &	4,065,788,034,000 &	663,772,725,000 &	15,807,052,800 & - & - \\
13 & 465,947,254,322,550 & 316,234,143,225 &	27,937,693,008,225 &	176,135,053,594,500 &	205,330,075,750,800 &	53,465,297,886,000 &	2,750,056,709,400 &	12,843,230,400 & - \\
14 & 22,187,548,315,807,876 & 7,905,853,580,625 &	835,382,740,292,100 &	6,628,818,891,444,750 &	10,359,374,548,723,200 &	3,985,285,430,926,800 &	365,229,364,500,000 &	5,551,486,340,400 & - \\
15 & 1,140,812,140,619,151,104 & 213,458,046,676,875 &	26,603,123,656,984,876 &	260,094,334,918,918,496 &	527,417,382,642,900,800 &	283,763,693,046,833,984 &	41,365,877,823,270,000 &	1,349,887,731,192,000 &	4,382,752,374,000 \\
\hline
\end{tabular}
\end{adjustbox}
\caption{Numbers $e^1_{n,g}$ of small leaf-labeled simplex time-consistent galled trees with specified numbers of leaves ($n$) and galls ($g$). Entries $e^1_{n,g}$ are computed recursively from eq.~(\ref{eq:e1ngrec}) in Section \ref{sec:exactll1tcgt}.}
\label{table:e1ng}
\end{table}
\end{landscape}

\end{document}